\documentclass[11pt]{amsart}

\usepackage[utf8]{inputenc}
\usepackage[T1]{fontenc}
\usepackage[english]{babel}

\usepackage{amsthm}
\usepackage{amsmath}
\usepackage{amssymb}
\usepackage{mathrsfs}
\usepackage{bbold}
\usepackage{graphicx}
\usepackage{listings}
\usepackage{mathrsfs}
\usepackage{enumerate}
\usepackage{pict2e}
\usepackage{stmaryrd}
\usepackage{enumitem}

\usepackage[all,cmtip]{xy}

\usepackage{mathtools}

\usepackage{color}
\definecolor{marin}{rgb}   {0.,   0.3,   0.7} 
\definecolor{rouge}{rgb}   {0.8,   0.,   0.} 
\definecolor{sepia}{rgb}   {0.8,   0.5,   0.} 
\usepackage[colorlinks,citecolor=sepia,linkcolor=marin,
            bookmarksopen,
            bookmarksnumbered
           ]{hyperref}

\newtheorem{lemma}{Lemma}[section]
\newtheorem{theorem}[lemma]{Theorem}
\newtheorem{proposition}[lemma]{Proposition}
\newtheorem{corollary}[lemma]{Corollary}

\theoremstyle{definition}
\newtheorem{definition}[lemma]{Definition}
\newtheorem{remark}[lemma]{Remark}

\newcommand*{\transp}[2][-3mu]{\ensuremath{\mskip1mu\prescript{\smash{\mathrm t\mkern#1}}{}{\mathstrut#2}}}

\begin{document}

\title[Exact splitting for linear partial differential equations]{Exact splitting methods for semigroups generated by inhomogeneous quadratic differential operators}

\begin{abstract}
We introduce some general tools to design exact splitting methods to compute numerically semigroups generated by inhomogeneous quadratic differential operators. More precisely, we factorize these semigroups as products of semigroups that can be approximated efficiently, using, for example, pseudo-spectral methods. We highlight the efficiency of these new methods on the examples of the magnetic linear Schr\"odinger equations with quadratic potentials, some transport equations and some Fokker-Planck equations.
\end{abstract}

\author{Joackim Bernier}
\address{Institut  de  Math\'ematiques  de  Toulouse  ;  UMR5219,  Universit\'e  de  Toulouse  ;  CNRS,  Universit\'e  Paul Sabatier, F-31062 Toulouse Cedex 9, France} 
\email{joackim.bernier@univ-nantes.fr}

\maketitle

%we will probably remove this, but for the moment it is useful to browse the file
\setcounter{tocdepth}{1} 
\tableofcontents

% a faire
% Noether and change of coordinate citer HLW
% sp isomorphe a l'espace des formes quadratiques
% Dire plus tot que P c'est un morphisme mais pas dans les notations car c'est pas canonique
% justifier Schrodinger magnetique
\section{Introduction}

\subsection{Motivation}
Before presenting the general setting of this paper, let us motivate our splitting methods on a quick example. Assume that we aim at solving numerically the following evolution equation on $\mathbb{R}^{n}$
\begin{equation*}
\left\{ \begin{array}{cccl} \partial_t u(t,x) & =& - (|x|^2 - \Delta)\, u(t,x), & t\geq 0,\ x\in \mathbb{R}^n \\ 
				u(0,x) &=& u_0(x), & x\in \mathbb{R}^n,
\end{array} \right.
\end{equation*}
where $n\geq 1$ and $u_0$ is a smooth and well localized function on $\mathbb{R}^n$. Since the Hermite functions diagonalize the harmonic oscillator, a first natural approach, to reach spectral accuracy, consists in using a Fast Hermite Transform. However, if the solution has to be known in the space variables (for example on a grid) at each time step, such a method is quite costly. 

A cheaper classical method consists in splitting the harmonic oscillator. For example, here, a natural splitting would be the following Strang splitting
\begin{equation}
\label{strang_intro}
e^{-\delta_t (|x|^2 - \Delta)} \simeq e^{-\delta_t |x|^2/2}e^{\delta_t \Delta}e^{-\delta_t |x|^2/2}.
\end{equation}
At each time step, this method requires only a Fast Fourier Transform to solve the heat equation. Unfortunately, even if it reaches spectral accuracy in space, it is only second order in time. Furthermore, note that, here,  the only way to get classical splitting methods of higher order is to use complex time steps (see e.g. \cite{BCCM,CCDV}) and that the higher the order of the method is, the larger the number of FFTs per time-step is (and so the more costly the method is).

However, allowing the sub-time-steps of the splitting method to be nonlinear functions of $\delta_t$, the following splitting formula can be established (see subsection 7.4 of \cite{AB} for a proof)
\begin{equation}
\label{split_exact_harm_osci}
 e^{-\delta_t (|x|^2-\Delta)}= e^{-\frac12 \tanh(\delta_t)  |x|^2}e^{\frac12 \sinh(2\delta_t)\Delta} e^{-\frac12 \tanh(\delta_t) |x|^2}.
\end{equation}
Note that, contrary to the Strang splitting \eqref{strang_intro}, this factorization is \emph{exact}: there is no remainder term. Consequently, it is much more accurate than \eqref{strang_intro} and the time step can be taken quite large (the only possible restrictions coming from the spatial discretization). Furthermore, it is as cheap as \eqref{strang_intro} to compute.

Roughly speaking, in this paper, we explain how and why the evolution equations associated with a large class of operators called \emph{inhomogeneous quadratic differential operators}, can be solved by some exact splitting methods. These exact splittings are similar to \eqref{split_exact_harm_osci}: they provide numerical methods of infinite order (spectral in space and exact in time) using only a small number of FFTs. Sometimes the coefficients of these splittings are not given by an explicit formula depending on $\delta_t$ (like in \eqref{split_exact_harm_osci}), nevertheless, in any case, we provide an iterative method to compute them efficiently. 

We warn the reader that the example of the exact splitting \eqref{split_exact_harm_osci} may be somehow misleading about our methods. They are not just a generalization of the classical splitting where the sub-time-steps would become nonlinear functions of $\delta_t$ (as in \eqref{split_exact_harm_osci}). Most of the time, the exact splitting requires the computation of solutions of evolution equations we would not compute in a classical splitting method. For example, to solve the Fokker--Planck equation

\begin{equation*}
\left\{ \begin{array}{cccl} \partial_t u(t,x,v) + v\partial_x u(t,x,v) &=& \partial_v (v+ \partial_v) u(t,x,v), &  x,v\in \mathbb{R}, \ t\geq 0,  \\ 
				u(0,x,v) &=& u_0(x,v), & x,v\in \mathbb{R},
\end{array} \right.
\end{equation*}
with an exact splitting, we have to solve some Schr\"odinger equations (see Proposition \ref{prop_NFP} for details).

\subsection{General context and result}
We consider the problem of the numerical resolution by splitting methods of linear partial differential equations of the form
\begin{equation}
\label{linear_pde}
\left\{ \begin{array}{cccl} \partial_t u(t,x) & =& - p^w u(t,x), & t\geq 0,\ x\in \mathbb{R}^n \\ 
				u(0,x) &=& u_0(x), & x\in \mathbb{R}^n
\end{array} \right.
\end{equation}
where $n\geq 1$, $u_0\in L^2(\mathbb{R}^n)$ and $p^w$ is an inhomogeneous quadratic differential operator acting on $L^2(\mathbb{R}^n)$. 

First, let us precise what we hear by \emph{inhomogeneous quadratic differential operator}. This terminology is used (for example by H\"ormander in \cite{Hor95}) to denote the Weyl quantization of a polynomial function of degree $2$ or less. More precisely, if $p$ is a polynomial function on $\mathbb{C}^{2n}$ of degree $2$ or less (called \emph{symbol}) whose decomposition in coordinates is given by
$$
p(X) = \transp{X} Q X + \transp{Y} X + c
$$
where $X=\transp{(x_1,\dots,x_n,\xi_1,\dots,\xi_n)}$, $Q$ is a symmetric matrix of size $2n$ with complex coefficients, $Y\in \mathbb{C}^{2n}$ is a vector and $c\in \mathbb{C}$ is a constant, the Weyl quantization of $p$ is the operator acting on $L^2(\mathbb{R}^n)$ defined by
$$
p^w = \transp{\begin{pmatrix} x \\ -i\nabla \end{pmatrix}}Q\begin{pmatrix} x \\ -i\nabla \end{pmatrix} + \transp{Y} \begin{pmatrix} x \\ -i\nabla \end{pmatrix} + c.
$$
Note that, for such polynomials, this formula coincides with the usual (and much more general) formula defining the Weyl quantization through oscillatory integrals (see e.g., Section 18.5 in \cite{Lars_book} or Chapter 1 in \cite{Rodino_book}).

The Weyl quantization allows in many situations the deduction of some properties of the operator from properties of its symbol. For example, as stated in the following proposition, if the real part of $p$ is bounded below on $\mathbb{R}^{2n}$ then \eqref{linear_pde} is globally well posed. 
\begin{proposition}
\label{prop_semigroup_exists} If $p$ is a polynomial of degree $2$ or less on $\mathbb{C}^{2n}$ whose real part is bounded below on $\mathbb{R}^{2n}$ then $-p^w$ generates a strongly continuous semigroup on $L^2(\mathbb{R}^n)$ denoted by $(e^{-t p^w})_{t\in \mathbb{R}_+}.$
\end{proposition}
This proposition is very classical and relies on the Hille-Yosida Theorem. For example, a proof is given by H\"ormander in \cite{Hor95} (pages 425-426) when $p$ is quadratic, but its proof can clearly be extended to the case where $p$ is inhomogeneous (see e.g. Theorem 4.7 in \cite{Hor95}).

We recognize that the class of equations given by \eqref{linear_pde}, where $p$ is a polynomial of degree $2$ or less, may seem too elementary to require the use of specific methods to solve them. However, it is usual to have to solve them as sub-steps of splitting methods for more sophisticated equations. Thus it is crucial to have robust methods to compute them very efficiently. Furthermore, for many of these models, in some relevant regimes, their dynamics are the leading part of the dynamics. So, it is crucial to be able to compute them with as much accuracy as possible. 

For example, the linear part of some nonlinear Schr\"odinger equations describing some rotating Bose--Einstein condensates (see e.g. \cite{Bader,Bao}) is of the form \eqref{linear_pde} where
$$
p^w = ( -i |\xi|^2 - i V(x) - i Bx \cdot \xi    )^w = i \Delta  -  B x \cdot \nabla  - i V(x).
$$
where $V:\mathbb{R}^n\to \mathbb{R}$ is a quadratic external potential and $B$ is a real skew symmetric matrix of size $n$ associated with a constant external magnetic field. The formalism of \eqref{linear_pde} allows also to consider transport equations associated with affine vectors fields. Indeed, if $B$ is a square real matrix of size $n$ and $y\in \mathbb{R}^n$ then
$$
(iB x\cdot \xi +i y\cdot \xi + \mathrm{Tr}B)^w=(B x + y)\cdot \nabla.
$$
Even if such transport equations are essentially trivial, their resolution is required to compute, using splitting methods, the solutions of many kinetic equations (e.g. the Vlasov-Maxwell equations with a constant external magnetic field, see \cite{JCC,cef}).
Finally, equations like \eqref{linear_pde} can also describe some phenomena of diffusions. For example, the generalized Ornstein--Uhlenbeck operators are of the form
$$( \xi \cdot A \xi + iBx\cdot \xi +   x\cdot Rx +  \mathrm{Tr}B )^w = -\nabla\cdot A\nabla + Bx \cdot \nabla + x\cdot Rx  $$
where $A,R$ are some real nonnegative symmetric matrices of size $n$ and $B$ is a real matrix $n$. Note that these operators include Fokker--Planck and Kramers--Fokker--Planck operators (see e.g. \cite{DHL,HSH}).

Some of the equations of the form \eqref{linear_pde} are easy to solve numerically using pseudo-spectral methods. For example, to solve the heat equation or to compute a shear, it is enough to do some Fast Fourier Transforms. Similarly, in the spirit of the splitting methods, it is not very costly to solve successively some of these equations. More precisely, let us define, in this context, the operators we consider as easily computable using standard pseudo-spectral methods.
\begin{definition}
\label{def_comput} An operator acting on $L^2(\mathbb{R}^n)$ can be \emph{computed by an exact splitting} if it can be factorized as a product of operators of the form 
$$e^{ \alpha \partial_{x_j}},\ e^{ i \alpha x_j },\ e^{ i a(\nabla)},\ e^{ i a(x) },\ e^{ \alpha x_k  \partial_{x_j} },\ e^{ -b(x)},\ e^{ b(\nabla)},\ e^{\gamma}$$ where $\alpha\in \mathbb{R},\gamma\in \mathbb{C},a,b : \mathbb{R}^n \to \mathbb{R}$ are some real quadratic forms, $b$ is nonnegative, $j,k\in \llbracket 1,n \rrbracket$ and $k\neq j$. As usual, $a(\nabla) $ (resp. $b(\nabla)$) denotes the Fourier multiplier associated with $-a(\xi)$ (resp. $-b(\xi)$), i.e. $a(\nabla) = (-a(\xi))^w$.
\end{definition}
Note that if an operator can be computed by an exact splitting then it is bounded. The following theorem justifies why we focus on splitting methods for semigroups generated by inhomogeneous quadratic differential operators.

\begin{theorem}
\label{thm_universal}
 If $p$ is a polynomial of degree $2$ or less on $\mathbb{C}^{2n}$ whose real part is bounded below on $\mathbb{R}^{2n}$ then $e^{-p^w}$ can be computed by an exact splitting.
\end{theorem}

\begin{remark} Actually, we prove a slightly stronger result : the quadratic forms associated with $a$ and $b$ in the definition of Definition \ref{def_comput} can be chosen diagonal. Nevertheless, it is not necessarily relevant from a numerical point of view because the sub-steps required to diagonalize the quadratic forms can be costly.
\end{remark}

This theorem is proved in the subsection \ref{sub_proof_thm_universal} of the Appendix. Unfortunately, its proof does not provide an efficient way to split semigroups (minimizing, for example, the number of sub-steps or the number Fast Fourier Transforms required to approximate the exponentials). Nevertheless, as illustrated in Section \ref{sec_app}, on many examples, paying attention to the particular structure of each semigroup, we are able to design optimized exact splittings.

This work can be considered as the theoretical part of a more general study. Indeed, a second work \cite{JCL}, written in collaboration with Nicolas Crouseilles and Yingzhe Li, deals with the implementation of these methods and compares numerically their efficiency, their accuracy and their qualitative properties with respect to the existing methods. We also couple these exact splitting methods with classical methods in order to solve some nonlinear equations and some non quadratic linear equations.

\medskip

\paragraph{\emph{Outline of the work}} In Section \ref{sec_eslode}, we develop the notion of exact splitting in a more general framework and specify its links with the classical splittings. It will naturally lead to a general theorem to design efficient exact splittings for many linear ordinary differential equations. In Section \ref{sec_FIO}, we explain how the theory of the Fourier Integral Operators developed by H\"ormander in \cite{Hor95} can be used to transform exact splittings of linear ordinary differential equations into exact splittings of semigroups. And, finally, in Section \ref{sec_app}, we apply the results of the previous sections to obtain some efficient exact splittings for the magnetic linear Schr\"odinger equations with quadratic potentials, some transport equations and some Fokker--Planck equations. 

\medskip

\paragraph{\emph{Acknowledgment}}
The author thanks P. Alphonse, N. Crouseilles and Y. Li for many enthusiastic discussions about this work. Research of the author was supported by ANR project NABUCO, ANR-17-CE40-0025.

\medskip

\paragraph{\emph{Notations and conventions}} Let us define some classical notations used in this paper.
\begin{itemize}
\item  To get convenient notations, most of the time, we denote by $t$ (instead of $\delta_t$) the time-step of our methods.  
\item $I_n$ denotes the identity matrix on $\mathbb{R}^n$ and $J_{2n}$ denotes the matrix of the canonical symplectic form of $\mathbb{R}^{2n}$, i.e.
$$
J_{2n} := \begin{pmatrix} & I_n \\
 -I_n & 
\end{pmatrix}.
$$
\item By convention, the empty spaces in the matrix notations refer to coefficients equal to zero (see for example the definition of $J_{2n}$ above).
\item If $A$ is a matrix $\transp{A}$ denotes the transpose of $A$.
\item If $\mathbb{K}\in \{\mathbb{R},\mathbb{C}\}$, $M_{n}(\mathbb{K})$ denotes the algebra of the square matrices of size $n$ and $\mathrm{S}_n(\mathbb{K})=\{ M \in M_{n}(\mathbb{K}), \ \transp M =M \}$ denotes the vector space of the symmetric matrices.
\item For $\mathbb{K}\in \{\mathbb{R},\mathbb{C}\}$, we will use the following classical groups of matrices
$$
\begin{array}{lll} \mathrm{GL}_{n}(\mathbb{K}) = \{ M \in M_{n}(\mathbb{K}), \ M \ \mathrm{is} \ \mathrm{invertible} \} & \mathrm{SL}_{n}(\mathbb{K}) = \{ M \in M_{n}(\mathbb{K}), \ \det M =1 \} \\
\mathrm{Sp}_{2n}(\mathbb{K}) = \{ M \in M_{2n}(\mathbb{K}), \ \transp M J_{2n} M =J_{2n} \} & \mathrm{O}_{n}(\mathbb{R}) = \{ M \in M_{n}(\mathbb{R}), \ \transp{M}M=I_n \}.
\end{array}
$$
and their associated Lie algebras
$$
\begin{array}{lll} \mathfrak{gl}_{n}(\mathbb{K}) = M_n(\mathbb{K}) & \mathfrak{sl}_{n}(\mathbb{K}) = \{ M \in M_{n}(\mathbb{K}), \ \mathrm{Tr}M=0 \} \\
\mathfrak{sp}_{2n}(\mathbb{K}) = \{ J_{2n}Q, \ Q\in \mathrm{S}_{2n}(\mathbb{K})  \} & \mathfrak{so}_{n}(\mathbb{R}) = \{ M \in M_{n}(\mathbb{R}), \ \transp{M} = - M  \}.
\end{array}
$$
where $\mathrm{Tr}$ denotes the trace and the Lie bracket is formally defined through the relation $[A,B]:=AB-BA$.
\item A real Lie algebra of matrices is a sub-Lie-algebra of $\mathfrak{gl}_n(\mathbb{R})$ for some $n\geq 0$.
\item We equip the space of the polynomials of degree $2$ or less on $\mathbb{C}^{2n}$ of its structure of Lie algebra induced by the canonical Poisson bracket. This one being defined for two polynomials $p_1,p_2$ on $\mathbb{C}^{2n}\equiv \mathbb{C}^{n}_x\times \mathbb{C}^{n}_{\xi}$ by
$$
\{p_1,p_2\} := \sum_{j=1}^n \partial_{\xi_j} p_1 \partial_{x_j} p_2 - \partial_{x_j} p_1 \partial_{\xi_j} p_2.
$$
Note that if $\mathbb{K}\in \{\mathbb{R},\mathbb{C}\}$, the space of the quadratic forms on $\mathbb{K}^{2n}$ is a Lie algebra naturally isomorphic to $\mathfrak{sp}_{2n}(\mathbb{K})$.
\item If $\mathfrak{g}$ is a Lie algebra, $\mathrm{ad}:\mathfrak{g}\to \mathfrak{g}$ denotes its adjoint representation, i.e. 
$$\forall x,y\in \mathfrak{g}, \ \mathrm{ad}_{x}y := [x,y].$$
\item We consider the natural action of the analytic functions on $M_{n}(\mathbb{C})$ defined by the holomorphic functional calculus (see VII-3 of \cite{Dunford_book} for details). By abuses of notations, if $f$ is analytic on a domain $\Omega\subset \mathbb{C}$ and $M\in M_n(\mathbb{C})$ has its spectrum included in $\Omega$ then $(f(z))(M)$ is just an other way to denote $f(M)$.  Note that, if $f:\mathbb{D}(z_c,\rho) \to \mathbb{C}$ is an analytic function on the complex disk of center $z_c$ and radius $\rho>0$ and $M\in M_{n}(\mathbb{C})$ has its spectrum included in $\mathbb{D}(z_c,\rho)$ then $f(M)$ can be defined by the convergent series
$$
f(M) = \sum_{k\in \mathbb{N}} \frac{f^{(k)}(z_c)}{k!} (M-z_c I_n)^k.
$$
\item $\mathcal{S}(\mathbb{R}^n)$ denotes the Schwartz space on $\mathbb{R}^n$.
\item In a non-commutative setting, the notation $\prod$ to denote a product can be quite ambiguous. In this paper, we adopt the following natural convention. If $I$ is a totally ordered finite set and $(g_j)_{j\in I} \in M^I$, where $M$ is a monoid\footnote{ i.e. a set equipped with an associative binary operation and an identity element.}, then
$$
\prod_{j\in I} g_j := g_{\iota_1} \dots g_{\iota_{\sharp I}}
$$
where $\iota$ is the increasing bijection from $\llbracket 1, \sharp I \rrbracket$ to $I$.
\end{itemize}

\section{Exact splitting methods}
\label{sec_eslode}

The original setting of the splitting methods (see e.g \cite{HLW}) consists in considering  linear equations\footnote{Note that as usual this formalism include nonlinear equations, since it is enough to consider the transport equations they generate.} of the form
\begin{equation}
\label{equ_flou}
\left\{ \begin{array}{lll} \partial_t u = Lu = L^{(1)} u + \dots + L^{(k)} u \\
				u(t=0) = u_0.
\end{array}\right.
\end{equation}
whose solution is denoted $u(t)=\exp(tL)u_0$ and such that 
\begin{center}
$\exp(tL^{(1)}),\dots,\exp(tL^{(k)})$ are \emph{nicer} than $\exp(tL)$.
\end{center}
 In this context, \emph{nicer} usually means \emph{cheaper} to compute. Then the approximations of $u$ at times $t_n=n\delta_t$ are got compositing $n$ times an approximation $\Psi_{\delta_t}$ of $\exp(\delta_tL)$ where $\Psi_{\delta_t}$ is a composition of operators of the form $\exp(\alpha_{\sigma_j} \delta_t L^{(\sigma_j)})$ with $1\leq \sigma_j\leq k$ and $\alpha_j\in \mathbb{R}$. The most classical methods are the Lie splitting where
$$
\Psi_{\delta_t}^{(Lie)} = \exp(\delta_t L^{(1)} )\dots \exp(\delta_t L^{(k)})
$$
and the Strang splitting
$$
\Psi_{\delta_t}^{(Strang)} = \exp(\frac{\delta_t}2 L^{(1)} )\dots  \exp(\frac{\delta_t}2 L^{(k-1)}) \exp(\delta_t L^{(k)})  \exp(\frac{\delta_t}2 L^{(k-1)}) \dots \exp(\frac{\delta_t}2 L^{(1)}).
$$
These methods are respectively of order $1$ and $2$. It means that they provide respectively approximations of order $\delta_t$ and $\delta_t^{2}$ of $u(t_n)$. Similar methods can be derived to obtain splitting methods of arbitrarily high order. However, note that, the higher the order is, the higher the number of step is and so the higher the cost of the method is. Furthermore, the only way to get methods of order higher than $2$ where the $\alpha_{\sigma_j}$ are nonnegative is to allow the $\alpha_{\sigma_j}$ to be complex. In this case it is possible to design schemes of high order where the real parts of the $\alpha_{\sigma_j}$ are nonnegative (see \cite{BCCM,CCDV}). Note that this is crucial when irreversible equations are considered.

The setting of the exact splitting is the same as the setting of the classical ones. Nevertheless, we have to assume that  
\begin{center}
$\exp(tL^{(1)}),\dots,\exp(tL^{(k)})$ are nicer than $\exp(tL)$ \emph{because} $L^{(1)},\dots,L^{(k)}$ belong to some vector spaces $E_1,\dots,E_k$.
\end{center}
 This assumption is natural because usually the $\exp(tL^{(j)})$ are nice due to a particular structure of  each $L^{(j)}$ (e.g. nilpotent, diagonal...). Consequently, the idea of the exact splitting consists in looking for $\Psi_{\delta_t}$ as a product of operators of the form $\exp( \delta_t L^{(\sigma_j)}_{j,\delta_t})$ with $1\leq \sigma_j\leq k$, $L^{(\sigma_j)}_{j,\delta_t}\in E_{\sigma_j}$ and to ask that $\Psi_{\delta_t}$ is exactly equal to $\exp(\delta_t L)$. In other words, we are looking for a factorization of $\exp(\delta_t L)$ as a product of operators of the form $\exp( \delta_t L^{(\sigma_j)}_{j,\delta_t})$ with $L^{(\sigma_j)}_{j,\delta_t}\in E_{\sigma_j}$.

\begin{remark}
 To avoid any possible confusion note that $\exp( \delta_t L^{(\sigma_j)}_{j,\delta_t})$ does not denote the solution of a non-autonomous equation but the solution of $\partial_t u = L^{(\sigma_j)}_{j,\delta_t} u$ at time $t=\delta_t$. 
\end{remark}

Actually, the existence of an exact splitting can be seen as an inverse problem with respect to the classical \emph{backward error analysis} of splitting methods. In the context of the splitting method, this analysis is realized through  the \emph{Baker-Campbell-Hausdorff formula} (see \cite{HLW}). It states that
\begin{equation}
\label{BEA_BCH}
\exp(\alpha_{\sigma_1} \delta_t L^{(\sigma_1)})\dots \exp(\alpha_{\sigma_\ell} \delta_t L^{(\sigma_\ell)}) = \exp(\delta_t \, \Omega_{\delta_t,L,\alpha,\sigma})
\end{equation}
where $\Omega_{\delta_t,L,\alpha,\sigma}	$ can be expanded in powers of $\delta_t$ and the operator associated with the power $\delta_t^{n}$ is obtained as a linear combination of $n$ Lie brackets of the $L^{(j)}$. For example, we have
$$
\Omega_{\delta_t,L,\alpha,\sigma} = \sum_{j=1}^\ell \alpha_{\sigma_j}  L^{(\sigma_j)} + \frac{\delta_t}2 \sum_{i<j}  \alpha_{\sigma_j}\alpha_{\sigma_i} [L^{(\sigma_i)} , L^{(\sigma_j)} ] + \mathcal{O}(\delta_t^2)
$$
where $[L^{(\sigma_i)} , L^{(\sigma_j)} ] = L^{(\sigma_i)} L^{(\sigma_j)}  -L^{(\sigma_j)} L^{(\sigma_i)}$. Note that, in general, the formula \eqref{BEA_BCH} and the expansion of $\Omega_{\delta_t,L,\alpha,\sigma}$ have to be understood in the sense of the formal series in $\delta_t$. Nevertheless, if $L^{(1)},\dots L^{(k)}$ belong to a real Lie algebra of matrices then these expansions converge when $\delta_t$ is small enough (see e.g. \cite{HLW}). 

Now if we consider a product of operators of the form $\exp( \delta_t L^{(\sigma_j)}_{j,\delta_t})$ with $1\leq \sigma_j\leq k$, $L^{(\sigma_j)}_{j,\delta_t}\in E_{\sigma_j}$, where the $E_j$ are some vector subspaces of a same real Lie algebra of matrices then the Baker-Campbell-Hausdorff formula states that it is of the form 
\begin{equation}
\label{omega_tilde}
\exp(\alpha_{\sigma_1} \delta_t L^{(\sigma_1)}_{1,\delta_t})\dots \exp(\alpha_{\sigma_\ell} \delta_t L^{(\sigma_\ell)}_{\ell,\delta_t}) = \exp(\delta_t \widetilde{\Omega}_{\delta_t,L,\sigma}),
\end{equation}
where $\widetilde{\Omega}_{\delta_t,L,\sigma}$ admits an expansion similar to the expansion of $\Omega_{\delta_t,L,\alpha,\sigma}$.
Consequently, to get an exact splitting method we just have to design $L^{(\sigma_\ell)}_{j,\delta_t}$ in order to cancel all the Poisson brackets in the expansion of $ \widetilde{\Omega}_{\delta_t,L,\sigma}$, i.e. we have to solve
\begin{equation}
\label{eq_to_solve}
\widetilde{\Omega}_{\delta_t,L,\sigma} = \sum_{j=1}^k L^{(j)}.
\end{equation}
Since it can be shown that $\widetilde{\Omega}_{\delta_t,L,\sigma}$ is smooth with respect to $\delta_t$ and $(L^{(\sigma_j)}_{j,\delta_t})_{1\leq j\leq \ell}$, it is natural to try to solve \eqref{eq_to_solve} with the Implicit Function Theorem. 

In this way, we prove the following theorem for which we have many applications. From now, to get convenient notations, we denote $t$ instead of $\delta_t$.  
\begin{theorem} \label{thm_split_Lie} Let $m$ be a positive integer and $\mathfrak{b}_1,\dots, \mathfrak{b}_m,\mathfrak{s}$ be some subspaces of a real Lie algebra of matrices such that $\mathfrak{b}_1,\dots, \mathfrak{b}_m$ are complementary. If $b_{\star}=b_{\star,1}+\dots+b_{\star,m}\in \mathfrak{b}_1\oplus \dots \oplus \mathfrak{b}_m= \mathfrak{b}$  is such that 
\begin{equation}
\label{the_assumption}
( \mathfrak{b}_1\oplus \dots \oplus \mathfrak{b}_m)+ \mathrm{ad}_{b_{\star}}(\mathfrak{s}) \ \ \mathrm{ is \ a\ real\ Lie \ algebra}
\end{equation}
then there exist $t_0>0$ and an analytic function $t\in (-t_0,t_0) \mapsto (b_t,s_t)\in \mathfrak{b}\times \mathfrak{s}$ such that $b_0=b_{\star}$ and
\begin{equation}
\label{split_ex_Lie}
\forall t\in (-t_0,t_0), \ e^{t b_{\star}} = e^{-t  s_t} e^{tb_{t,1}} \dots e^{tb_{t,m}}e^{t  s_t}.
\end{equation}
\end{theorem}

\begin{remark}
Before proving this theorem, let us do some remarks. 
\begin{itemize}[leftmargin=*]
\item In most of the applications, $\mathfrak{s}$ can be chosen as a one of the $\mathfrak{b}_j$. Consequently, assuming that $\mathfrak{s} = \mathfrak{b}_1$, the notations of this theorem are consistent with the notations previously introduced in this section through the identifications $\delta_t=t$, $k=m$, $E_j = \mathfrak{b}_j$, $b_\star = L^{(1)}+\dots+L^{(m)}$, $\ell=k+2$, $L_{1,\delta_t}^{(1)}=-s_t, L_{j,\delta_t}^{(j-1)}=b_{t,j-1}$, $L_{\ell,\delta_t}^{(1)}= s_t$ where $j\in \llbracket 2,m+1\rrbracket$.
\item Since the proof of this theorem relies on the Implicit Function Theorem, the coefficients of the exact splitting (i.e. $s_t$ and $b_t$) can be efficiently computed by an iterative method.
Unfortunately, this method require some notations introduced in the proof of Theorem \ref{thm_split_Lie} to be presented. Consequently, it is introduced just after the proof.
\item In practice, as we will see in Section \ref{sec_app}, it may be useful to use Theorem \ref{thm_split_Lie} just to determine a priori the form of an exact splitting and then to get analytically the associated coefficients (using for example a formal computation software).
\item The assumption \eqref{the_assumption} of Theorem \ref{thm_split_Lie} is a bit too strong. The optimal assumption seems that $\widetilde{\Omega}_{\delta_t,L,\sigma}$ (defined implicitly by \eqref{omega_tilde}) belongs to the vector space defined in \eqref{the_assumption} for all $L$ and $\delta_t$. Paying attention to the expansion of $\widetilde{\Omega}_{\delta_t,L,\sigma}$ given by the Baker-Campbell-Hausdorff formula, it essentially means that we do not need to ask this space to contain Lie brackets of elements belonging to a same space $\mathfrak{b_j}$. Note that, such a generalization would be necessary to justify \emph{a priori} the form of the exact splitting of the Example 2.3 of \cite{AB} for the Kramer--Fokker--Planck operator.
\end{itemize}
\end{remark}

\begin{proof}[Proof of Theorem \ref{thm_split_Lie}] Naturally, following the assumption of the theorem, we introduce the real Lie algebra of matrices, denoted $\mathfrak{g}$ and defined by
\begin{equation}
\label{dec_of_g}
\mathfrak{g} = ( \mathfrak{b}_1\oplus \dots \oplus \mathfrak{b}_m)+ \mathrm{ad}_{b_{\star}}(\mathfrak{s}).
\end{equation}

Applying the Baker Campbell Hausdorff formula, we get a neighborhood of the origin in $ \mathfrak{b}\times  \mathfrak{s} $, denoted $V$, and an analytic function $F:V\to  \mathfrak{g}$ such that for all real $t$ and all $ (b,s)\in  \mathfrak{b}\times \mathfrak{s}$ such that $  t(b,s)\in V$, we have
$$
e^{-t s}e^{t b_1} \dots e^{t b_m}e^{t s}=\\ \exp \left( t b+ t^2 [b,s] + \frac{t^2}2 \sum_{1\leq i<j\leq m} [b_i,b_j] +
 t^3F(tb,ts) \right).
$$
Consequently, we aim at solving the equation
\begin{equation}
\label{the eq to solve}
 b_{\star} = \\ b + t  \mathrm{ad}_b(s) + \frac{t}2 \sum_{1\leq i<j\leq m} [b_i,b_j] +
 t^2 F(tb,ts).
\end{equation}

To solve this equation we have to introduce some notations. First, from the decomposition \eqref{dec_of_g} of $\mathfrak{g}$, we deduce naturally that there exists a complementary space to $ \mathfrak{b}$ in $\mathfrak{g}$ denoted $\mathfrak{r}$, a subspace of $\mathfrak{s}$ denoted $\mathfrak{s}'$ and a vector space isomorphism $\Psi : \mathfrak{s}' \to \mathfrak{r}$ such that
\begin{equation}
\label{def_Psi}
\Psi = \Pi_{\mathfrak{r}} \circ \mathrm{ad}_{b_{\star}}
\end{equation}
where $\Pi_{\mathfrak{b}},\Pi_{\mathfrak{r}}$ are the canonical projections associated with the decomposition 
\begin{equation}
\label{the_dec}
\mathfrak{g}=\mathfrak{b}\oplus \mathfrak{r}.
\end{equation}
Then, let $s_{\star} \in \mathfrak{s}'$ be defined by
\begin{equation}
\label{def_srstar}
s_{\star}=  -\frac12 (\Psi^{-1}\circ \Pi_{\mathfrak{r}}) \left( \sum_{1\leq i<j\leq m} [b_{\star,i},b_{\star,j}] \right),
\end{equation}
let $V_{\star}$ be a neighborhood of $(b_{\star},s_{\star})$ in $\mathfrak{b}\times \mathfrak{s}'$ and let $t_{\star}>0$ be such that 
$$(-t_{\star},t_{\star}) V_{\star}\subset V.$$
So, to solve \eqref{the eq to solve}, we are going to apply the Implicit Function Theorem in $(0,b_{\star},s_{\star})$ to the function
$$
G:\left\{ \begin{array}{ccc}(-t_{\star},t_{\star}) \times V_{\star} &\to& \mathfrak{b}\times \mathfrak{s}' \\
(t,b,s) &\mapsto& (G_{\mathfrak{b}},G_{\mathfrak{s}'})(t,b,s)  \end{array}   \right.
$$  
where 
$$
 G_{\mathfrak{b}}(t,b,s)= b-b_{\star}+  t \ \Pi_{\mathfrak{b}} R(t,b,s) .
$$
with
$$
R(t,b,s) = [b,s]  + \frac12  \sum_{1\leq i<j\leq m} [b_i,b_j] + t F(tb,ts),
$$
and
$$
\Psi G_{\mathfrak{s}'}(t,s^{(m)},b,s^{(r)}) = \Pi_{\mathfrak{r}} \circ \mathrm{ad}_{b} s +  \frac12 \Pi_{\mathfrak{r}}   \sum_{1\leq i<j\leq m} [b_i,b_j]  \\+ t\ \Pi_{\mathfrak{r}}  F(tb,ts).
$$
 First, observe that, by construction of $G$, if $G(t,b,s)=0$ then $b,s$ is a solution of \eqref{the eq to solve}. Indeed, by construction, the equation \eqref{the eq to solve} can be rewritten as 
 $$( G_{\mathfrak{b}} + t \Psi G_{\mathfrak{s}'})(t,b,s) = 0.$$

  Then observe that, by construction of $s_{\star}$ we have 
 $$G(0,b_{\star},s_{\star})=0.$$
 Consequently, since $G$ is clearly an analytic function, to conclude the proof applying the Implicit Function Theorem, we just have to prove that the partial differential of $G$ with respect to $(b,s)$ in $(0,b_{\star},s_{\star})$ is invertible. Indeed, using a natural matrix representation, this differential is
 $$
 \begin{pmatrix}	 \mathrm{I}_{\mathfrak{b} }& \\
			 L_{\mathfrak{b},\star} & \mathrm{I}_{\mathfrak{s}' }
 \end{pmatrix}
 $$
 where  $L_{\mathfrak{b},\star}: \mathfrak{b} \to \mathfrak{s}'$ is an explicit linear map depending on $b_{\star}$ and $\mathfrak{s}'$. This matrix being triangular, it is clearly invertible and its invert is
  $$
 \begin{pmatrix}	 \mathrm{I}_{\mathfrak{b} }& \\
			-L_{\mathfrak{b},\star} & \mathrm{I}_{\mathfrak{s}' }
 \end{pmatrix}
 $$
\end{proof}

Now, let us present an iterative method to determine the coefficients $b_t$ and $s_t$ given by the Theorem \ref{thm_split_Lie}.
\begin{proposition} \label{prop_iter_gen} Assume that the assumption of Theorem \ref{thm_split_Lie} is satisfied, let $\Psi$ be defined by \eqref{def_Psi}, $s_{\star}$ be defined by \eqref{def_srstar} and $\Pi_{\mathfrak{b}},\Pi_{\mathfrak{r}}$ be the projections canonically associated with the decomposition \eqref{the_dec}.  \\
There exists $t_1\in (0,t_0)$, such that for all $t\in (-t_1,t_1)$ the sequence $(b_{t}^{(k)},s_t^{(k)})_{k\in \mathbb{N}}$  if well defined by induction as follow :

Initially, we have  $b_{t}^{(0)} = b_{\star}$ and $s^{(0)}_{t}=s_{\star}$ and for $k\geq 0$

\begin{equation}
\label{naive_induction}
\left\{ \begin{array}{llllllr}
b_{t}^{(k+1)} &=& b_{t}^{(k)} &+& b_{\star} &-& \Pi_{\mathfrak{b}} g^{(k)}\\
s^{(k+1)}_{t} &=& s^{(k)}_{t} && &-&t^{-1}\Psi^{-1} \Pi_{\mathfrak{r}} g^{(k)}\\
\end{array}
\right.
\end{equation}
where
$$
g^{(k)}=t^{-1}\log( e^{-t  s_t^{(k)}} e^{tb_{t,1}^{(k)}} \dots e^{tb_{t,m}^{(k)}}e^{t  s_t^{(k)}}).
$$
 Furthermore, there exist $C>0$ independent of $t$ and $k$ such that for all $k\geq 0$ and all $t\in (-t_1,t_1)$, we have
$$
  | b_{t}^{(k)} - b_t | + | s^{(k)}_t - s_t |  \leq C \ 2^{-k}.
$$
\end{proposition}
Note that, here, as usual, the $\log$ function denotes the principal determination of the matrix logarithm. We have chosen to present an iterative method as elementary as possible. However, computing the Lie brackets associated with the operators $L_{\mathfrak{s},\star}$ and $L_{\mathfrak{b},\star}$ it would be possible to get a natural iterative method whose convergence rate would be $\tau^k$ instead of $2^{-k}$ with $\tau$ a linear function of $|t|$.

\begin{proof}[Proof of Proposition \ref{prop_iter_gen}] It follows from the proof of the theorem \ref{thm_split_Lie} that the sequence is defined by a relation of the kind
$$
(b_{t}^{(k+1)} ,s^{(k+1)}_{t})= W_t(b_{t}^{(k)} ,s^{(k)}_{t}) 
$$
where for all $t\in (-t_0,t_0)$, $W_t$ is a smooth function on a same ball $\mathcal{B}$ of center $(b_{\star},s_{\star})$ and $t\mapsto W_t$ is smooth. Consequently, to prove that the sequence is well defined if $|t|$ is small enough  and that $W_t$ has a fix point we just have to prove that $W_t$ is a contraction mapping for a well chosen norm. Indeed, it follows of the proof of the theorem \ref{thm_split_Lie} that the differential of $W_0$ in $(b_{\star},s_{\star})$, denoted $\mathrm{d} W_0(b_{\star},s_{\star})$,  is 
  $$
 \begin{pmatrix}  \mathrm{0}_{\mathfrak{b} }& \\
			 -L_{\mathfrak{b},\star} & \empty \mathrm{0}_{\mathfrak{s}' }
 \end{pmatrix}.
 $$
Thus, since it is nilpotent, applying for example the Lemma 5.6.10 of \cite{Horn_book}, we get a matrix norm $\| \cdot \|_{\star}$ such that 
$$\|\mathrm{d} W_0(b_{\star},s_{\star})\|_{\star}\leq 1/4.$$
Now, since $W$ is smooth on $(-t_0,t_0)\times \mathcal{B}$ and $W_0(b_{\star},s_{\star})=(b_{\star},s_{\star}) $, we deduce of the mean value inequality that there exist $t_2\in (0,t_0)$ and a ball $\mathcal{B}_{\star}$, for the norm $\| \cdot \|_{\star}$, centered in $(b_{\star},s_{\star})$ such that for all $t\in (-t_2,t_2)$, $W_t$ is $1/2$ Lipschitzian on $\mathcal{B}_{\star}$, and $W_t \mathcal{B}_{\star} \subset \mathcal{B}_{\star}$. Consequently, $W_t$ has an unique fix point in $\mathcal{B}_{\star}$ and the sequence $(b_{t}^{(k)} ,s^{(k)}_{t})$ converges to this fix point with the rate $2^{-k}$.

Finally, we conclude this proof observing that, by construction and continuity, there exists $t_1\in(0,t_2)$ such that for all $t\in (-t_1,t_1)$, $(b_t,s_t)$ belongs to $\mathcal{B}_{\star}$ and is a fix point of $W_t$.
\end{proof}

\section{Exact classical-quantum correspondance}
\label{sec_FIO}

As we have seen in Section \ref{sec_eslode}, exact splittings can be designed (using for example Theorem \ref{thm_split_Lie}) to solve linear ordinary differential equations. We aim at designing exact splittings to solve some partial differential equations. So, first, it is natural to focus on linear transport equations. Indeed, the formula of the characteristics 
\begin{equation}
\label{charac_formula}
e^{t \ Bx\cdot \nabla} u = u\circ e^{t B}
\end{equation}
where $u\in L^2(\mathbb{R}^n)$ and $B$ is square matrix of size $n$, provides a natural way to transform an exact splitting at the level of the linear ordinary differential equations into an exact splitting  at the level of the linear transport equations.

For example, we have the following exact splitting for the two-dimensional rotations
\begin{equation}
\label{magic_formula}
\exp \left(
\begin{array}{cc}
0 & \tan (\theta/2) \\
0 & 0
\end{array}
\right)
\exp \left(
\begin{array}{cc}
0 & 0 \\
-\sin \theta & 0
\end{array}
\right)
\exp \left(
\begin{array}{cc}
0 & \tan (\theta/2) \\
0 & 0
\end{array}
\right) 
 = \exp( \theta J_2 ) ,  
\end{equation}
where $\theta \in (-\pi,\pi)$. At the level of the associated transport equation, this formula can be written
$$
e^{t (x_2\partial_{x_1} -  x_1\partial_{x_2}) } = e^ {\tan(t/2) x_2 \partial_{x_1}  } e^ {- \sin(t) x_1 \partial_{x_2}  } e^ {\tan(t/2) x_2 \partial_{x_1}  }.
$$
This factorization is very useful to compute rotations since, using semi-Lagrangian methods, it only requires one dimensional interpolations instead of a two dimensional interpolations as we could expect. The formula is well known in image processing and has been used for decades to rotate images (see e.g. \cite{paeth}). Two recent papers deal with the applications of this decomposition for the numerical resolution of kinetic equations (see \cite{Ameres,JCC}). This factorization has also been extended to compute $3$ dimensional rotations with only one dimensional interpolations (see e.g. \cite{3drotation,3drotation2}). Note that, in Subsection \ref{sub_transp}, we extend this kind of decomposition in any dimension and for more general transforms (including rotations and dilatations).

The inhomogeneous quadratic differential operators and the semigroups they generate enjoy some strong and specific properties providing a more general way to transform an exact splitting at the level of the linear ordinary differential equations into an exact splitting at the semigroup level.

Indeed, usually, the Lie bracket of two pseudo-differential operators is also a pseudo-differential operator. Its symbol is given by the \emph{Moyal bracket} of the symbols. In many applications, the Moyal bracket admits a natural expansion whose leading part\footnote{in some specific sense depending on the problem.} is given by the \emph{Poisson bracket} of the symbols (see e.g. \cite{Lars_book}). Furthermore, in the particular case of the inhomogeneous quadratic operators all the higher order terms vanish and we have
\begin{equation}
\label{eq_bracket}
[q_1^w,q_2^w]=-i\{q_1,q_2\}^w
\end{equation}
where $q_1,q_2$ are some polynomials of degree $2$ or less on $\mathbb{C}^{2n}$. 

The formula \eqref{eq_bracket} is especially useful to realize some changes of unknowns in order to put some operators in normal form (see e.g. \cite{BGMR} for an application to the reducibility that can also be seen as a factorization of exponentials). Note that, applying formally the Baker--Campbell--Hausdorff formula at the level of the operator, \eqref{eq_bracket} suggests a way to transform an exact splitting at the level of the linear ordinary differential equations into an exact splitting at the level of the semigroups generated by inhomogeneous quadratic forms. We refer the reader to the beginning of the Section $3$ of \cite{AB} for details about this heuristic. It is used in \cite{Chin,Blanes,Bader} to design some exact splitting methods to solve linear Schr\"odinger equations with harmonic potentials and rotating terms. The result suggested by these formal computations is made rigorous in Proposition \ref{prop_key} below, and relies on the representation of the semigroups as \emph{Fourier Integral Operators} (detailed in Theorem \ref{merci_lars} below).

To present this notion, we need to introduce some basic notations and associated properties.
\begin{definition} $T$ is a \emph{non-negative complex symplectic linear bijection} on $\mathbb{C}^{2n}$, and we denote $T \in \mathrm{Sp}_{2n}^+(\mathbb{C})$, if $T\in \mathrm{Sp}_{2n}(\mathbb{C})$ and 
$$
\forall X\in \mathbb{C}^{2n}, \ \transp{\overline{X}}\transp{\overline{T}} (-iJ_{2n}) T X- \transp{\overline{X}}(-iJ_{2n})X \in \mathbb{R}_+.
$$
\end{definition}
Note that this set is naturally equipped with a structure of monoid (i.e. it is stable by multiplication and the identity belongs to $\mathrm{Sp}_{2n}^+(\mathbb{C})$).

\begin{definition} If $q$ is a quadratic form on $\mathbb{C}^{2n}$ then its \emph{Hamiltonian flow} at time $t\in \mathbb{R}$, denoted by $\Phi_t^q$, is the flow at time $t$ of the linear ordinary differential equation
\begin{equation}
\label{hamilt_eq}
 \dfrac{\mathrm{d}}{\mathrm{d} t}z = - i J_{2n} \nabla q(z)
\end{equation}
where $\nabla q = \transp(\partial_{x_1}q,\dots,\partial_{x_{n}}q,\partial_{\xi_1}q,\dots,\partial_{\xi_{n}}q)$.
\end{definition}
Note that, in particular, the linear ordinary differential equation \eqref{hamilt_eq} can be solved using an exponential and thus we have 
\begin{equation}
\label{faudrait_surtout_pas_loublier}
\Phi_t^q = e^{-2itJ_{2n}Q}
\end{equation}
where $Q$ is the matrix of $q$. The following proposition summarizes some elementary properties of these Hamiltonian flows that will be used all along this paper.
\begin{proposition}
\label{prop_elementary_ham} Let $q_1,q_2,q$ be some quadratic forms on $\mathbb{C}^{2n}$ and $T\in \mathrm{Sp}_{2n}(\mathbb{C})$ then the following properties holds
\begin{enumerate}[label=\roman*)]
\item $\forall t\in \mathbb{R}$, $\Phi^q_t\in \mathrm{Sp}_{2n}(\mathbb{C})$,
\item $\{q_1,q_2\}=0$ $\iff$ $\forall t\in \mathbb{R},$  $\Phi_t^{q_1}\Phi_t^{q_2} = \Phi_t^{q_1+q_2}$,
\item $\forall t\in \mathbb{R}$, $T^{-1} \Phi_t^q T = \Phi_t^{q\circ T} $.
\end{enumerate}
\end{proposition}

Usually, in this context, the second property is called \emph{Noether's theorem} because it is also equivalent to have $q_1 \circ \Phi_t^{q_2} = q_1$ for all $t\in \mathbb{R}$. Proofs of these properties can be found, for example, in a nonlinear context, in \cite{HLW}.

 The following classical lemma (whose proof is recalled in the subsection \ref{sub_proof_lem_pos} of the Appendix) links naturally the two previous definitions.
\begin{lemma} \label{lem_pos} If $q$ is a quadratic form on $\mathbb{C}^{2n}$ such that $\Re q$ is nonnegative on $\mathbb{R}^{2n}$ then 
$$\forall t\geq 0, \ \Phi_t^q \in \mathrm{Sp}_{2n}^+(\mathbb{C}).$$
\end{lemma}

The following theorem, proved by H\"ormander in \cite{Hor95} (Thm 5.12 and Prop 5.9), is the main tool we use to realize exact splittings.
\begin{theorem}[H\"ormander \cite{Hor95}] \label{merci_lars} There exists a map
$$
\mathscr{K}:\mathrm{Sp}_{2n}^+(\mathbb{C}) \to \overline{ \mathcal{B} }/\mathbb{U}_2
$$
where $\mathbb{U}_2 = \{ +1,-1 \}$ and $\mathcal{B}$ is the unit ball of $\mathscr{L}(L^2(\mathbb{R}^n))$, the algebra of the bounded operators on $L^2(\mathbb{R}^n)$, 
which is monoid morphism, i.e.
 $$
 \forall S,T\in \mathrm{Sp}_{2n}^+(\mathbb{C}), \ \mathscr{K}(S T)  = \mathscr{K}(S) \mathscr{K}(T) \quad and \quad \mathscr{K}(I_{2n}) = \pm \mathrm{id}_{L_2}.
$$
 Furthermore, if $q$ is a quadratic form on $\mathbb{C}^{2n}$ such that $\Re q$ is nonnegative on $\mathbb{R}^{2n}$ then
 $$
 \mathscr{K}(\Phi_t^q) = \pm e^{-t q^w}.
 $$
\end{theorem} 
In \cite{Hor95}, $\mathscr{K}$ is defined through an explicit but heavy formula that is not relevant for us here. An operator of the form $\mathscr{K}(T)$ with $T\in \mathrm{Sp}_{2n}^+(\mathbb{C})$ is called a \emph{Fourier Integral Operators}. It provides a natural way to transform an exact splitting at the level of the linear ordinary differential equations into an exact splitting at the level of the semigroups generated by quadratic differential operators (up to an argument of continuity to remove the uncertainty of sign).

To the best of our knowledge, the idea to use the Fourier integral operators to get an exact splitting has been introduced by Paul Alphonse and the author in \cite{AB}. We aimed at characterizing the regularizing effects of the semigroups generated by non-selfadjoint quadratic differential operator. Our splitting provides a quite explicit representation of the polar decomposition of this semigroup. Note that it is the decomposition we would get applying Theorem \ref{thm_split_Lie} and Theorem \ref{merci_lars} with $\mathfrak{s}=\{0\}$, $\mathfrak{b}_1 = i\, \mathfrak{sp}_{2n}(\mathbb{R})$ and $\mathfrak{b}_2 = \mathfrak{sp}_{2n}(\mathbb{R})$.

In order to give a corollary of Theorem \ref{merci_lars} well suited to get exact splittings for semigroups generated by inhomogeneous quadratic differential operators, we have to introduce a last elementary notation. If $p: \mathbb{C}^{2n} = \mathbb{C}_{x}^{n}\times \mathbb{C}_{\xi}^{n} \to \mathbb{C}$ is a polynomial of degree $2$ or less that we write in coordinates as
$$
p = \transp{X} Q X + \transp{Y} X + c
$$
where $X=\transp{(x_1,\dots,x_n,\xi_1,\dots,\xi_n)}$, $Q\in \mathrm{S}_{2n}(\mathbb{C})$, $Y\in \mathbb{C}^{2n}$ and $c\in \mathbb{C}$, then $\mathbb{P} p :\mathbb{C}^{2n+2}=\mathbb{C}^{n+1}_{x}\times \mathbb{C}^{n+1}_{\xi} \to \mathbb{C} $ is the complex quadratic form defined by
$$
\mathbb{P} p = \transp{X} Q X +\transp{Y} X  x_{n+1}+  c x_{n+1}^2.
$$
\begin{proposition}
\label{prop_key} Let $p_{1,t},\dots,p_{m+1,t}:\mathbb{C}^{2n} \to \mathbb{C}$ be some polynomials  of degree $2$ or less depending continuously on $t\in [0,t_0]$ for some $t_0>0$, whose real part is uniformly bounded below on $\mathbb{R}^{2n}$ and satisfying
\begin{equation}
\label{the only thing we really have to prove}
\forall t\in [0,t_0],\  \Phi_t^{\mathbb{P} p_{1,t}} \dots   \Phi_t^{\mathbb{P} p_{m,t}} = \Phi_t^{\mathbb{P} p_{m+1,t}}
\end{equation}
then we have
\begin{equation}
\label{prop_split_ex}
\forall t\in [0,t_0],\ e^{-t p_{1,t}^w} \dots e^{-t p_{m,t}^w} = e^{-t p_{m+1,t}^w}.
\end{equation}
\end{proposition}
The proof is given at the end of this section. Let us mention that this proposition is also a  corollary of the Theorem 2.3 of \cite{Joe}. Note that if the polynomials are homogeneous (i.e. if they are some quadratic forms) then the $\mathbb{P}$ can be removed. Furthermore, as a corollary of the proof, if we do not assume that the polynomials are uniformly bounded below and that the polynomials depend continuously on $t$, then \eqref{prop_split_ex} holds up to an uncertainty of sign.

It may also be interesting to have exact splittings to compute evolution operators generated by \emph{non-autonomous} inhomogeneous quadratic differential operators (like it is done for the non-autonomous linear magnetic Schr\"odinger equations in \cite{Bader}). Here, we chose for conciseness to do not consider the non-autonomous case. Nevertheless, let us mention that it would be possible to generalize Proposition \ref{prop_key} to deal with non-autonomous equations using the generalization of the Theorem \ref{merci_lars} proven in \cite{Karel}. With such a generalization, the exponential of matrices become naturally the solutions of non-autonomous linear ordinary differential equations (which can be studied similarly using Magnus expansions, see \cite{HLW}).

In order to illustrate Proposition \ref{prop_key}, let us give an elementary application (Section \ref{sec_app} being devoted to more sophisticated applications). Indeed, the formula \eqref{magic_formula} used to compute rotations can clearly be extended analytically for $\theta \in -2it$ where $t\in \mathbb{R}$. Since, up to a transposition, this formula can be written as
$$
\Phi_t^{ (\tanh t)/(2t) \  x^2} \Phi_t^{ (\sinh 2t)/(2t) \  \xi^2} \Phi_t^{ (\tanh t)/(2t) \  x^2} = \Phi_t^{|x|^2+|\xi|^2}
$$
we deduce of the Proposition \ref{prop_key}  that the exact splitting \eqref{split_exact_harm_osci} for the harmonic oscillator (presented in the introduction) holds.

Before focusing on the proof of Proposition \ref{prop_key}, let us present the following Lemma that is useful in practice to establish the factorization \eqref{the only thing we really have to prove}. It expresses the \emph{triangular nature} of the equation \eqref{the only thing we really have to prove}. Its proof is postponed in subsection \ref{pippo} of the appendix.
\begin{lemma}
\label{lem_in_fact_its_triangular}
Let $p_{1},\dots,p_{m}$ be some polynomials  of degree $2$ or less on $\mathbb{C}^{2n}$ whose decompositions are 
$$
p_j= q_j +\ell_j+ c_j
$$
where $q_j$ quadratic forms, $\ell_j$ are linear forms and $c_j$ are complex numbers. Then we have
\begin{equation}
\label{the_aff_fact}
\Phi_1^{\mathbb{P} p_{1}} \dots   \Phi_1^{\mathbb{P} p_{m}} = \Phi_1^{\mathbb{P} p_{m+1}}
\end{equation}
if and only if
\begin{equation}
\label{terrible_system}
\left\{ \begin{array}{lll} \Phi_1^{q_{1}} \dots   \Phi_1^{ q_{m}} = \Phi_1^{q_{m+1}} \\
\displaystyle \sum_{j=1}^m \ell_j \circ \big( \Upsilon_{q_j} \prod_{k>j} \Phi_1^{q_{k}} \big) = \ell_{m+1} \circ \Upsilon_{q_{m+1}} \\
\displaystyle \sum_{j=1}^m c_j + \kappa_j + \sigma_j = c_{m+1} + \kappa_{m+1}
\end{array}
\right.
\end{equation}
where, denoting $Q^{(j)}$ the matrix of $q_j$ and $L^{(j)}\in \mathbb{C}^{2n}$ the matrix of $\ell_j$
$$
\Upsilon_{q_j} = \left( \frac{e^z-1}z \right)(-2iJ_{2n} Q^{(j)}),  \ \ \ \kappa_j = \sum_{k\in \mathbb{N}} \frac{4^k}{(2k+3)!} \ q_j( (J_{2n} Q^{(j)})^k J_{2n} \transp L^{(j)}),
$$
$$
 \sigma_j = -\frac{i}2 \sum_{p=1}^{j-1} L^{(p)}\Upsilon_{q_p} \left( \prod_{k=p+1}^{j-1}  \Phi_1^{q_{k}} \right) \Upsilon_{q_j} J_{2n}  \transp L^{(j)}.
$$
\end{lemma}
In this lemma, it is relevant to note that if $q_1,\dots,q_{m+1}$ satisfy the first equation of \eqref{terrible_system} then the second equation is just a linear system with respect to $\ell_1,\dots,\ell_{m+1}$ and that if the two first equations are satisfied then the last equation is linear with respect to $c_1,\dots,c_{m+1}$.

In order to prove Proposition \ref{prop_key}, we introduce some technical lemmas that will be crucial.

\begin{lemma}
\label{lem_struc_pos} If $p:\mathbb{R}^{2n} \to \mathbb{R}$ is a real polynomial of degree two or less being bounded below then there exists a nonnegative real quadratic form $q:\mathbb{R}^{2n}\to \mathbb{R}_+$, $Y\in \mathbb{R}^{2n}$ and $c\in \mathbb{R}$ such that $p$ can be written as
\begin{equation}
\label{the good form}
p(X) = q(X-Y) + c
\end{equation}
where $X=(x_1,\dots,x_n,\xi_1,\dots,\xi_n)$.
\end{lemma}
\begin{proof}[Proof of Lemma \ref{lem_struc_pos}]
Naturally $p$ can be written in coordinates as
$$
p(X) = \transp{X}QX+\transp{Z}X + p(0)
$$
where $Q\in S_{2n}(\mathbb{R})$ is the matrix of $q$ and $Z \in \mathbb{R}^{2n}$. Now realizing a canonical factorization, we observe that it is equivalent to prove that $p$ can be written as \eqref{the good form} and that $Z\in \mathrm{Im}\ Q$. 
In other words, we just have to prove that
$$\forall X_0 \in \mathrm{Im} \ Q^{\perp}, \ \transp{X_0} Z =0.$$
But since if $X_0\in \mathrm{Im}\ Q^{\perp}$ we have
$$
p(\lambda X_0) = \lambda \transp{X_0} Z + p(0)
$$
and we know, by assumption, that $\lambda \mapsto p(\lambda X_0)$ is bounded below then we deduce that $\transp{ X_0} Z=0$, which conclude this proof.
\end{proof}
\begin{corollary} 
\label{cor_proof_proj_pos}
If $p:\mathbb{R}^{2n} \to \mathbb{R}$ is a real polynomial of degree two or less being bounded below then we have
$$
\mathbb{P}(p-\inf_{\mathbb{R}^{2n}} p) \geq 0
$$
\end{corollary}
\begin{proof}[Proof of Corollary \ref{cor_proof_proj_pos}]
Applying Lemma \ref{lem_struc_pos}, $p$ can be written as
$$
p(X) = q(X-Y) + c.
$$
So, by definition, we have 
$$
\mathbb{P}(p-c)= q(X-x_{n+1}Y). 
$$
Thus, since $q$ is nonnegative, $\mathbb{P}(p-c)$ is also nonnegative and $c$ is the infimum of $p$.
\end{proof}

\begin{lemma} 
\label{lemma_restri}
If $\psi \in \mathcal{S}(\mathbb{R}^{n+1})$ and $p$ is a polynomial of degree $2$ or less on $\mathbb{C}^{2n}$ whose real part is nonnegative on $\mathbb{R}^{2n}$ then $\Re \mathbb{P}p$ is nonnegative on $\mathbb{R}^{2n+2}$,  $e^{-(\mathbb{P}p)^w} \psi \in  \mathcal{S}(\mathbb{R}^{n+1})$ and we have
$$
(e^{-(\mathbb{P}p)^w} \psi)_{| \mathbb{R}^n \times \{1\}} = e^{- p^w} \psi_{| \mathbb{R}^n \times \{1\}}
$$
\end{lemma} 

\begin{proof}[Proof of Lemma \ref{lemma_restri}]

First, observe that Corollary \ref{cor_proof_proj_pos} proves directly that $\Re\mathbb{P}p$ is nonnegative. Consequently, the semigroup generated by $-(\mathbb{P}p)^w$ is well defined (see e.g. Proposition \ref{prop_semigroup_exists}).

Applying the Theorem 4.2 of \cite{Hor95}, we know that 
$$
\forall t\in [0,1], \ \phi(t)=e^{-t(\mathbb{P}p)^w} \psi \in  \mathcal{S}(\mathbb{R}^{n+1})
$$
and that $\phi \in \mathcal{C}^{\infty}(\mathbb{R}^+_t\times \mathbb{R}^{n+1})$ and satisfies 
$$
\forall t\in \mathbb{R}_+, \ \partial_t \phi = -(\mathbb{P}p)^w  \phi.
$$
Now decomposing $p$ by homogeneity as
\begin{equation*}
p = q + \ell + c
\end{equation*}
where $q:\mathbb{R}^{2n} \to \mathbb{C}$ is a complex valued quadratic form, $\ell:\mathbb{R}^{2n} \to \mathbb{C}$ is a complex valued linear form and $c\in \mathbb{C}$, we have
$$
\forall t\in \mathbb{R}^+, \ \partial_t \phi = -(q^w + x_{n+1} \ell^w + cx_{n+1}^2)   \phi.
$$
Since here we deal with smooth functions, this relation can be evaluated in $x_{n+1}=1$. Thus we get
$$
\forall t\in \mathbb{R}^+, \ \partial_t \rho = -(q^w +  \ell^w + c)  \rho = -p^w \rho
$$ 
where $\rho = \phi_{| \mathbb{R}^+_t\times \mathbb{R}^{n}\times \{1\}}$. Consequently by definition of the semigroup $\exp(-tp^w)$, we have
$$
\rho = \exp(-tp^w) \rho(0,\cdot) =\exp(-tp^w) \psi_{|\mathbb{R}^{n}\times \{1\}}.
$$
Finally, evaluating this last relation for $t=1$, we get the desired relation.
\end{proof}

Using these two lemmas, we give the following proof of the Proposition \ref{prop_key}.

\begin{proof}[Proof of Proposition \ref{prop_key}]
In order to obtain a factorization of the semigroup using Theorem \ref{merci_lars}, we have to deal with quadratic forms having a nonnegative real part. So, since the polynomials are uniformly bounded below, we apply Corollary \ref{cor_proof_proj_pos} to get a constant $c>0$ such that
\begin{equation}
\label{obs1}
\forall j\in \llbracket 1,m+1\rrbracket,\forall t\in  [0,t_0],  \ \Re \ \mathbb{P} (p_{j,t} + c) \geq 0.
\end{equation}

Since the real part of $\mathbb{P} (p_{j,t}+c) = \mathbb{P} p_{j,t} + c x_{n+1}^2$ is nonnegative, we know by lemma \ref{lem_pos} that $$
\forall j\in \llbracket 1,m+1\rrbracket,\forall t\in  [0,t_0], \ \Phi_t^{\mathbb{P} p_{j,t}+cx_{n+1}^2} \in \mathrm{Sp}_{2n+2}^+(\mathbb{C}).
$$
Then observing that, by construction $\mathbb{P} p_{j,t}$ does not depend on $\xi_{n+1}$, it commutes with $x_{n+1}^2$, i.e.
$$
\{ \mathbb{P} p_{j,t} , x_{n+1}^2\} = 0.
$$
where $\{ \cdot,\cdot\}$ stands for the usual Poisson bracket associated with the canonical symplectic form on $\mathbb{R}_x^{n+1}\times \mathbb{R}_{\xi}^{n+1}$.
Consequently, applying the Noether theorem\footnote{note that here it could be proven more elementarily, applying Lemma \ref{lem_in_fact_its_triangular}.}, their Hamiltonian flows commute, i.e. for all $j\in \llbracket 1,m\rrbracket$ and all $t\in  [0,t_0]$ we have
$$
 \Phi_t^{\mathbb{P} p_{j,t}+cx_{n+1}^2} = \Phi_t^{\mathbb{P} p_{j,t}} \Phi_t^{cx_{n+1}^2} \quad \textrm{and} \quad \Phi_t^{\mathbb{P} p_{m+1,t}+cmx_{n+1}^2} = \Phi_t^{\mathbb{P} p_{m+1,t}} \Phi_t^{cmx_{n+1}^2},
$$
and we have
$$
\forall t\in  [0,t_0], \ \Phi_t^{\mathbb{P} p_{1,t} + c x_{n+1}^2} \dots   \Phi_t^{\mathbb{P} p_{m,t}+ c x_{n+1}^2} = \Phi_t^{\mathbb{P} p_t + m c x_{n+1}^2}
$$

Now, we can apply the Theorem \ref{merci_lars} to deduce that for all $t\in  [0,t_0]$, there exists $\varepsilon_t \in \{+1,-1\}$ such that
$$
\forall t\in [0,t_0],\ e^{-t (\mathbb{P}p_{1,t})^w - t c x_{n+1}^2} \dots e^{-t (\mathbb{P}p_{m,t})^w - t c x_{n+1}^2} = \varepsilon_t e^{-t (\mathbb{P} p_{m+1,t})^w - t cm  x_{n+1}^2}.
$$

To get the desired factorization, we have to check that $\varepsilon_t=1$ and to prove that this relation can be evaluated at $x_{n+1}=1$. 

First, we focus on the sign and we observe that if $t=0$ all the exponentials are equal to the identity so we have $\varepsilon_0 =1$. Thus, since $[0,t_0]$ is connected, we just have to prove that $t\mapsto \varepsilon_t$ is continuous to deduce that $\varepsilon_t=1$ for all $t\in [0,t_0]$. Let $\phi\in \mathcal{S}(\mathbb{R}^n)\setminus \{0\}$ be a Schwarz function on $\mathbb{R}^n$ non identically equals to zero (for example a gaussian). Since $e^{-t (\mathbb{P} p_t)^w - t cm  x_{n+1}^2}$ is injective (see Theorem 2.1 and Corollary 7.9 of \cite{AB}), we now that 
$$
\forall t\in [0,t_0], \ \| e^{-t (\mathbb{P} p_{m+1,t})^w - t cm  x_{n+1}^2} \phi \|_{L^2} > 0.
$$
So we have
$$
\varepsilon_t = \langle \prod_{j=1}^m e^{-t (\mathbb{P}p_{j,t})^w - t c x_{n+1}^2} \phi, e^{-t (\mathbb{P} p_{m+1,t})^w - t cm  x_{n+1}^2} \phi \rangle_{L^2} \| e^{-t (\mathbb{P} p_{m+1,t})^w - t cm  x_{n+1}^2} \phi \|_{L^2}^{-2}.
$$
But it follows for the Theorem 4.2 of \cite{Hor95} that $t\mapsto e^{-t (\mathbb{P}p_{1,t})^w - t c x_{n+1}^2} \dots e^{-t (\mathbb{P}p_{m,t})^w - t c x_{n+1}^2} \phi$ and $t\mapsto e^{-t (\mathbb{P} p_{m+1,t})^w - t cm  x_{n+1}^2} \phi$ are continuous from $[0,t_0]$ to $\mathcal{S}(\mathbb{R}^n)$. Consequently, $t\mapsto \varepsilon_t $ is continuous as product of two continuous functions and we have $\varepsilon_t\equiv 1$, i.e.
\begin{equation}
\label{almost_good}
\forall t\in [0,t_0],\ e^{-t (\mathbb{P}p_{1,t})^w - t c x_{n+1}^2} \dots e^{-t (\mathbb{P}p_{m,t})^w - t c x_{n+1}^2} = e^{-t (\mathbb{P} p_{m+1,t})^w - t cm  x_{n+1}^2}.
\end{equation}

To conclude, we just have to prove that this relation can be evaluated in $x_{n+1}=1$. 

Let $\phi \in \mathcal{S}(\mathbb{R}^n)$ and choose $\psi \in  \mathcal{S}(\mathbb{R}^{n+1})$ such that
$$
\psi_{| \mathbb{R}^n \times \{1\}} = \phi.
$$
Applying $m+1$ times Lemma \ref{lemma_restri} to \eqref{almost_good}, we deduce that
$$
\forall t\in [0,t_0],\ e^{-t (p_{1,t})^w - t c } \dots e^{-t (p_{m,t})^w - t c }\phi = e^{-t  p_{m+1,t} ^w - t cm  }\phi.
$$
Thus, since it is clear that for all $t\in [0,t_0]$ and all $ j\in \llbracket 1,m\rrbracket$ we have
$$
 e^{-t (p_{j,t})^w - t c } =e^{-t (p_{j,t})^w} e^{- t c } \quad \textrm{and} \quad e^{-t (p_{m+1,t})^w - t mc } =e^{-t (p_{m+1,t})^w} e^{- t mc }
$$
we deduce that
$$
\forall t\in [0,t_0], \ e^{-t (p_{1,t})^w } \dots e^{-t (p_{m,t})^w }\phi = e^{-t  p_{m+1,t} ^w }\phi.
$$
Finally, since the semigroups are continuous on $L^2(\mathbb{R}^n)$ and that the previous relation holds for all $\phi\in \mathcal{S}(\mathbb{R}^n)$ that is dense in $L^2(\mathbb{R}^n)$, we get the desired factorization.
\end{proof}

\section{Applications}
In this section, we apply the results of the previous sections in order to get some efficient exact splitting methods for the magnetic linear Schr\"odinger equations with quadratic potentials, some transport equations and some Fokker--Planck equations.  Another paper \cite{JCL}, written in collaboration with Nicolas Crouseilles and Yingzhe Li, deals with the implementation of these methods and compares numerically their efficiency, their accuracy and their qualitative properties with respect to the existing methods. 

\label{sec_app}

\subsection{Transport equations}
\label{sub_transp} As we have seen through the formula \eqref{magic_formula} two dimensional rotations can be computed efficiently as products of shear transforms.

In fact, this kind of factorization is much more general since as a classical application of the Gaussian elimination algorithm, we know that each matrix of determinant $1$ is a product of \emph{shear matrices} \footnote{also called \emph{transvection matrices}.} (see e.g. Lemma 8.7 of \cite{Rotman_book}). Identifying as usual matrices with linear maps on $\mathbb{R}^n$, this factorization is written as follows.
\begin{proposition} \label{prop_sluniv} For all $G \in \mathrm{SL}_n(\mathbb{R})$, there exists $m\geq 0$, $\alpha \in \mathbb{R}^m,\ k,\ell\in \llbracket 1,n \rrbracket^m$ such that for all $j\in \llbracket 1,m \rrbracket$, $k_j\neq \ell_j$ and
$$
\forall u\in L^2(\mathbb{R}^n), \ u\circ G =  \exp(\alpha_1 x_{k_1}\partial_{x_{\ell_1}}) \cdots \, \exp(\alpha_m x_{k_m}\partial_{x_{\ell_m}}) u
$$
\end{proposition}
Note that in the context of the exact splitting, it is more natural to apply this result for $G = \exp(t B)$ where $B\in \mathfrak{sl}_n({\mathbb{R}})$ is a matrix whose trace vanishes.

If we focus on transforms associated with matrices of determinant different than $1$, we have to deal with one dimensional dilatations. Indeed, it is clear that each matrix $G\in  \mathrm{GL}_n(\mathbb{R})$ can be factorized as a product of a matrix in $\mathrm{SL}_n(\mathbb{R})$ and the diagonal matrix $\mathrm{diag}(1,\dots,1,\det G)$. The following proposition provides a way to deal with positive dilatations with pseudo-spectral methods.

\begin{proposition} 
\label{prop_dilat}
For all $\lambda>0$, we have
$$
\forall u\in L^2(\mathbb{R}), \ u (\lambda \ \cdot) =  \lambda^{-1/2} e^{ -i \varepsilon_\lambda \alpha_\lambda  \partial_x^2}  e^{-i \beta_\lambda  x^2} e^{i \varepsilon_\lambda \beta_\lambda \partial_x^2} e^{i \alpha_\lambda  x^2}u,
$$
where $\alpha_\lambda=\frac12 \sqrt{|\lambda^{-1} - 1|\lambda^{-1}}$, $\beta_\lambda=\frac12 \sqrt{|1-\lambda|}$, $\varepsilon_\lambda = 1$ if $\lambda \leq 1$ and $\varepsilon_\lambda = -1$ else.
\end{proposition}
\begin{proof}[Proof of Proposition \ref{prop_dilat}] We observe that if $t=\log \lambda$ then 
$$
u (\lambda \ \cdot) = e^{t x\partial_x} u = e^{t (ix\xi-\frac12)^w} u.
$$
Consequently, this proposition is a consequence of the proposition \ref{prop_key} and an elementary formal computation which can be checked, for example, with a formal computation software.
\end{proof} 
We do not provide similar formulas when $\lambda<0$ since we do not know if it may be useful for applications. Nevertheless, since we have (see e.g. Thm 2.2.3 of \cite{Rodino_book})
\begin{equation}
\label{fun_fun}
\forall u\in L^2(\mathbb{R}), \  u(-x) = -i  e^{i \frac{\pi}2 (x^2-\partial_x^2)}u(x),
\end{equation}
we note that, as a consequence of Theorem \ref{thm_universal}, such formulas exist. %if you think it is relevant I can do this

The factorization provided by the Gaussian elimination algorithm in Proposition \ref{prop_sluniv} is not, in general, the most efficient possible. The following proposition, that is an application of Theorem \ref{thm_split_Lie}, provides some efficient exact splittings generalizing the exact splittings for rotations \eqref{magic_formula}.
\begin{proposition}
 \label{prop_rot_gen} Let $M$ be a real square matrix of size $n\geq 1$ such that
\begin{equation}
\label{bobof_prop}
\left\{ \begin{array}{lll} \forall i, \ &M_{i,i} = 0 \\
				\exists i,\forall j  \neq i , & M_{i,j} \neq 0 
\end{array} \right.
\end{equation} 
then there exist $t_0>0$ and an analytic function $(y^{(\ell)},(y^{(k)})_{k\neq j},y^{(r)}):(-t_0,t_0)\to \mathbb{R}^{n\times (n+1)}$ satisfying
\begin{equation}
\label{penible_prop}
\left\{ \begin{array}{lll} y^{(\ell)}_i = y^{(r)}_i =0 \\
				 \forall k\neq i, \ y^{(k)}_k = 0
\end{array} \right.
\end{equation} 
  such that for all $t\in  (-t_0,t_0)$ we have
$$
e^{t Mx\cdot \nabla} = e^{t (y^{(\ell)}(t)\cdot x)\partial_{x_i}} \left(  \prod_{ k\neq i} e^{t (y^{(k)}(t)\cdot x)\partial_{x_k}}  \right)  e^{t (y^{(r)}(t)\cdot x)\partial_{x_i}}.
$$
\end{proposition}
For example, this proposition justifies a priori the form of the exact splitting used in  \cite{3drotation,3drotation2} to compute three dimensional rotations.

\begin{proof}[Proof of Proposition \ref{prop_rot_gen}] Let $B:=\transp{M}$ be the transpose of $M$. We are going to prove that there exists an analytic function $(y^{(\ell)},(y^{(k)})_{k\neq i},y^{(r)}):(-t_0,t_0)\to \mathbb{R}^{n\times (n+1)}$ satisfying \eqref{penible_prop} such that for all $t\in (-t_0,t_0)$, \begin{multline}
\label{not_fun}
e^{t B} = (I_n + t \ y^{(\ell)}(t)\otimes e_j) \left(  \prod_{ k\neq j}  (I_n + t \ y^{(k)}(t)\otimes e_k) \right)(I_n + t \ y^{(r)}(t)\otimes e_j)\\  = e^{ t  y^{(\ell)}(t)\otimes e_j} \left(  \prod_{ k\neq i} e^{ t  y^{(k)}(t)\otimes e_k}  \right)e^{ t  y^{(r)}(t)\otimes e_j}.
\end{multline}
Consequently, Proposition \ref{prop_rot_gen} will be proven since it is enough to apply the characteristic formula \eqref{charac_formula} to the transpose of \eqref{not_fun}. Note that the second equality in \eqref{not_fun} comes form the fact that the matrices are nilpotent of order $1$.

In order, to prove this factorization applying the Theorem \ref{thm_split_Lie}, we realize the change of unknown  $y^{(i)} = y^{(r)} + y^{(\ell)}$ and thus the first factor of \eqref{not_fun} becomes 
$$
e^{ t  y^{(\ell)}(t)\otimes e_j} = e^{ -t  y^{(r)}(t)\otimes e_j} e^{ t  y^{(i)}(t)\otimes e_j}.
$$

Up to an irrelevant permutation of indices, without loss of generality we assume that $i=1$. In order to apply the Theorem \ref{thm_split_Lie}, we define 
$$
\mathfrak{b}_k  = \{  y \otimes e_k \ | \ y \in \mathbb{R}^n \ \mathrm{ satisfies } \ y_k =0\}, \  \  \mathfrak{s} = \mathfrak{b}_1 \ \textrm{and} \ \ \  \mathfrak{b} =\bigoplus_{k=1}^n \mathfrak{b}_k 
$$
and $b_{\star,k} = Be_k \otimes e_k$ (i.e. $b_{\star} = B$).
Consequently, to prove this proposition applying the Theorem \ref{thm_split_Lie}, we just have to prove that its assumption \eqref{the_assumption} is satisfied. In our context, we are going to prove that
\begin{equation}
\label{vrai_truc_a_prouver}
\mathfrak{b} + \mathrm{ad}_{B}(\mathfrak{s})   = \mathfrak{sl}_n(\mathbb{R}).
\end{equation}

So, from now, we just focus on proving \eqref{vrai_truc_a_prouver}. But, since the inclusion $"\subset"$ is obvious, we just have to prove that
\begin{equation}
\label{plus_simple}
\left\{ \begin{array}{lll} \displaystyle
\mathfrak{b} \cap  \mathrm{ad}_{B}(\mathfrak{s}) = \{0\}, \\ \displaystyle
\dim \mathfrak{b}  +  \dim  \mathrm{ad}_{B}(\mathfrak{s})  = \dim  \mathfrak{sl}_n(\mathbb{R}).
\end{array} \right.
\end{equation}

Let $\Psi : \mathfrak{s}  \to \mathbb{R}^n$ be the linear map defined by
$$
\forall s \in \mathfrak{s}, \ \Psi(s) = \mathrm{diag} \circ \mathrm{ad}_B(s)
$$ 
where $\mathrm{diag}: \mathfrak{gl}_n(\mathbb{R}) \to \mathbb{R}^n$ is the natural map extracting the diagonal coefficients of a matrix.

Since $ \mathfrak{b}  = \mathrm{Ker} \ \mathrm{diag}$ is the space of the matrices with diagonal coefficients are equal to zero, the first relation of \eqref{plus_simple} is equivalent to have $\mathrm{Ker} \ \Psi = \{0\}$. Furthermore, if $\mathrm{Ker}\ \Psi = \{0\}$, then $\mathrm{ad}_B$ is injective on $\mathfrak{s}$ and we have
$$
\dim  \mathrm{ad}_{B}(\mathfrak{s}) = \dim \mathfrak{s} = n-1.
$$
Since, for all $k$, $\dim \mathfrak{b}_k = n-1$ and $(n+1)(n-1)= n^2-1 = \dim  \mathfrak{sl}_n(\mathbb{R})$, we deduce that if $\mathrm{Ker}\ \Psi = \{0\}$ then the second relation of \eqref{plus_simple} also holds.

Finally, we just have to verify that $\mathrm{Ker} \ \Psi = \{0\}$. Let $y \otimes e_1 \in \mathrm{Ker} \ \Psi$ where $y \in \mathbb{R}^n$ satisfies $y_1=0$. By assumption, if $j\neq 1$ then
$$
0 = \transp e_j [B,y \otimes e_1] e_j = - \transp e_j y \transp e_1B e_j = - y_j B_{j,1}. 
$$
Since, by assumption, $B_{j,1}\neq 0$ for $j\neq 1$, we deduce that $y = 0$, i.e. $\mathrm{Ker} \ \Psi = \{0\}$.
\end{proof}

\subsection{Schr\"odinger equations}

Let $v:\mathbb{R}^n \to \mathbb{R}$ be a real quadratic form and $B\in \mathfrak{so}_n(\mathbb{R})$ be a skew symmetric matrix. We aim at solving the following linear Schr\"odinger equation on $\mathbb{R}^n$
\begin{equation}
\label{schro_eq}
i\partial_t u(t,x) + \frac12 \Delta u(t,x) - v(x) u(t,x) + i Bx \cdot \nabla u(t,x) = 0.
\end{equation}
In this context, a diagonal quadratic form is defined as follow.
\begin{definition}
A \emph{quadratic form} is \emph{diagonal} on $\mathbb{R}^n$ if its matrix in the canonical basis is diagonal.
\end{definition}
The following theorem provides an optimized splitting method to solve \eqref{schro_eq}. Its proof is given at the end of this subsection.
\begin{theorem} \label{thm_splt_MS2} There exists some quadratic forms $v_t^{(r)}, a_t$ on $\mathbb{R}^{n}$, a strictly upper triangular matrix $U_t \in M_n(\mathbb{R})$,  a strictly lower triangular matrix $L_t \in M_n(\mathbb{R})$ and a diagonal quadratic form $v_t^{(\ell)}$ on $\mathbb{R}^{n}$, all depending analytically on $t\in (-t_0,t_0)$ for some $t_0>0$, such that for all $t\in (-t_0,t_0)$ we have
$$
e^{ i t ( \Delta/2 - v(x)) - t B x\cdot \nabla } =e^{-i t v^{(\ell)}_t(x)}  \left( \prod_{j=1}^{n-1} e^{- t (U_t x)_j \partial_{x_j}} \right) e^{i t a_t(\nabla)} \left( \prod_{j=2}^{n} e^{- t (L_t x)_j \partial_{x_j}} \right) e^{-i t v^{(r)}_t(x)} 
$$ 
where $a_t(\nabla)$ denotes the Fourier multiplier of symbol $-a_t(\xi)$ and $(U_t x)_j$ (resp. $(L_t x)_j$) the $j^{\mathrm{st}}$ coordinate of $U_t x$ (resp. $L_t x$).
\end{theorem}

The efficiency of this method is optimal, more precisely it is as cheap as a basic Lie splitting. Indeed, considering that the computational cost of this kind of method is proportional to the number of one dimensional Fast Fourier Transform required to implement it, the method of Theorem \ref{thm_splt_MS2} requires $2n$ 1d-FFTs\footnote{An inverse Fourier transform has the same cost as a direct one. For example, to compute the solution of the semigroup generated by the harmonic oscillator with the factorization \eqref{split_exact_harm_osci}, we need $2$ 1d-FFTs, one to go on the Fourier side and one other to come back.} which is the same as the elementary Lie splitting 
$$
e^{ i t ( \Delta/2 - v(x)) - t B x\cdot \nabla } \mathop{=}_{t\to 0} e^{-it v(x)} \prod_{j=1}^n e^{\frac{it}2 \partial_{x_j}^2 - (B x)_j\cdot \partial_{x_j}} + \mathcal{O}(t^2).
$$
Note that this method is more general and more efficient than the other existing exact splittings for \eqref{schro_eq}. For example, the splitting $(2)$ of Bader in \cite{Bader}, designed to solve \eqref{schro_eq} in dimension $n=2$, requires $6$ 1d-FFTS.

It seems to the author that there does not exist any simple explicit formula to compute the coefficients of the exact splitting given in Theorem \ref{thm_splt_MS2}. However, 
due to some particular properties of nilpotency, a simple and efficient iterative method is available to compute them. Indeed, identifying a quadratic form with its symmetric matrix in order we define, if $t$ is small enough, we define, by induction, the following sequences
$$
\left\{\begin{array}{lll} A_{t,k+1} &=& A_{t,k} + I_n/2 - \widetilde{A}_{t,k} \\
L_{t,k+1} &=& L_{t,k} + L - \widetilde{L}_{t,k}\\
				U_{t,k+1} &=& U_{t,k} + U - \widetilde{U}_{t,k} \\
				V_{t,k+1}^{(m)} &=& V_{t,k}^{(m)} + V - \widetilde{V}_{t,k}^{(m)} + \frac{t}2 [D_{t,k},B ]  + \frac{t^2}2 D_{t,k}^2
\end{array}  \right.
$$
where $(A_{t,0},L_{t,0}+U_{t,0},V_{t,0}^{(m)}) = (I_n/2,B,V)$, $L+U=B$ and
$$
\begin{pmatrix} 2 \widetilde{V}_{t,k}^{(m)}  & \transp{\widetilde{L}_{t,k}}+\transp{\widetilde{U}_{t,k}} + t D_{t,k}\\
				\widetilde{L}_{t,k}+\widetilde{U}_{t,k} + t D_{t,k} & 2\widetilde{A}_{t,k}	 \\
\end{pmatrix} = -t^{-1} J \log(P_{t,k})
$$
and
\begin{multline*}
P_{t,k} =  \left[ \prod_{j=1}^{n-1} \begin{pmatrix} I_n + t U_{t,k}^{(j)}&   \\ & I_n - t \transp{ U_{t,k}^{(j)}}  \end{pmatrix}\right]  \begin{pmatrix} I_n &  2 t A_{t,k} \\  & I_n \end{pmatrix} 
\left[ \prod_{j=2}^n \begin{pmatrix} I_n + t L_{t,k}^{(j)}&   \\ & I_n - t \transp{ L_{t,k}^{(j)}}  \end{pmatrix}\right] \\ \times
  \begin{pmatrix} I_n &   \\  -2 t V_{t,k}^{(m)}  & I_n \end{pmatrix}
\end{multline*}
with $L_{t,k}^{(j)} = (e_j \otimes e_j)L_{t,k}$, $U_{t,k}^{(j)} = (e_j \otimes e_j)U_{t,k}$ and $(e_1,\dots,e_n)$ the canonical basis of $\mathbb{R}^n$.

\begin{theorem}
\label{thm_iter}%cadarache 
There exists $\tau_0\in (0,t_0)$ such that if $0<|t|<\tau_0$ then the preceding sequences are well defined and
$$
|A_t - \widetilde{A}_{t,k} | + |L_t - \widetilde{L}_{t,k} | + |U_t - \widetilde{U}_{t,k} | + |V_t^{(\ell)} + \frac12 D_{t,k} | + |V_t^{(r)} - V_{t,k}^{(m)} -  \frac12D_{t,k} |  \leq \left(\frac{t}{\tau_0} \right)^k.
$$
\end{theorem}

We could apply directly Theorem \ref{thm_split_Lie} to prove Theorem \ref{thm_splt_MS2}. Indeed, by Proposition \ref{prop_key}, it is clearly enough to prove that there exist some quadratic forms $v_t^{(r)}, a_t$ on $\mathbb{R}^{n}$, a strictly upper triangular matrix $U_t \in M_n(\mathbb{R})$,  a strictly lower triangular matrix $L_t \in M_n(\mathbb{R})$ and a diagonal quadratic form $v_t^{(\ell)}$ on $\mathbb{R}^{n}$, all depending analytically on $t\in (-t_0,t_0)$ for some $t_0>0$, such that for all $t\in (-t_0,t_0)$
\begin{equation}
\label{New fact}
\Phi_t^{ i (\frac{|\xi|^2}2 +  v(x)  + B x\cdot \xi )  }
 =\Phi_t^{i  v^{(\ell)}_t(x)}  \left( \prod_{j=1}^{n-1} \Phi_t^{ i(U_t x)_j \xi_j} \right) \Phi_t^{i a_t(\xi)} \left( \prod_{j=2}^{n} \Phi_t^{ i(L_t x)_j \xi_j} \right) \Phi_t^{i  v^{(m)}_t(x)}  \Phi_t^{-i  v^{(\ell)}_t(x)} 
\end{equation}
where $v^{(r)}_t := v^{(m)}_t - v^{(\ell)}_t$. Furthermore note that all the factors of \eqref{New fact} are some real symplectic transforms belonging to $\mathrm{Sp}_{2n}(\mathbb{R})$ and that all the right hand side flows are exponentials of nilpotent matrices (it explains why we don't have to compute exponential of matrices in the iterative method). Nevertheless, due to some convenient algebraic cancelations, the iterative method described below is not exactly the one associated with the proof of Theorem \ref{thm_split_Lie}. Consequently, here it is easier to prove Theorem \ref{thm_splt_MS2} and Theorem \ref{thm_iter} simultaneously. However, the proofs being very similar to the proof of Theorem \ref{thm_split_Lie}, so it is done more informally.

Before giving the proof, let us just mention shortly the algebraic decomposition we would use in order to satisfy the assumption \eqref{the_assumption} of Theorem \ref{thm_split_Lie} to prove \eqref{New fact}. Recalling that $\mathfrak{sp}_{2n}(\mathbb{R})$ is isomorphic to the Lie algebra of the real quadratic forms (equipped with the canonical Poisson bracket), we would simply have to observe that all real quadratic form $q:\mathbb{R}^{2n}=\mathbb{R}_x^{n} \times \mathbb{R}_{\xi}^{n}  \to \mathbb{R}$ written 
$$
q(x,\xi) = a^{(q)}(\xi) + M^{(q)} x \cdot \xi + b^{(q)}(x)
$$
where $a^{(q)},b^{(q)}$ are real quadratic forms on $\mathbb{R}^n$ and $M^{(q)}\in M_{n}(\mathbb{R})$, can be decomposed as
\begin{multline*}
q(x,\xi) =  \sum_{j=1}^{n-1} (U_t^{(q)} x)_j \xi_j  + a^{(q)}(\xi) + \sum_{j=2}^{n} (L_t^{(q)} x)_j \xi_j + \left( b^{(q)}(x) -
 \frac{[D,B]x \cdot x}2 \right) \\
 + \left\{   \frac{|\xi|^2}2 +  B x\cdot \xi +v(x)  , \frac{\transp{x} D^{(q)}x}2   \right\}.
\end{multline*}
where $M^{(q)} = L^{(q)} + D^{(q)}+ U^{(q)} $ is the natural decomposition of $M^{(q)}$.

\begin{proof}[Proof of Theorem \ref{thm_splt_MS2} and Theorem \ref{thm_iter}] Similarly to the proof of Theorem \ref{thm_split_Lie}, we prove that the exact splitting can be obtained from the resolution of a nonlinear equation via the Implicit Function Theorem. Consequently, Theorem \ref{thm_iter} is nothing but the convergence of the natural iterative method associated with the Implicit Function Theorem.

In order to obtain \eqref{New fact}, we consider \eqref{New fact} as a nonlinear equation where $v_t^{(m)}, a_t, U_t, L_t,v_t^{(\ell)}$ are the unknowns. Applying the BCH formula, we get a family of real quadratic forms $q_t$ on $\mathbb{R}^{2n}$ (depending analytically on $v_t^{(m)}, a_t, U_t, L_t$)  such that if $t$ is small enough then
$$
\left( \prod_{j=1}^{n-1} \Phi_t^{ i(U_t x)_j \xi_j} \right) \Phi_t^{i a_t(\xi)} \left( \prod_{j=2}^{n} \Phi_t^{ i(L_t x)_j \xi_j} \right) \Phi_t^{i  v^{(m)}_t(x)} = \Phi_t^{i q_t}.
$$
Thus, since $\Phi_t^{-i  v^{(\ell)}_t(x)}$ is symplectic, applying Proposition \ref{prop_elementary_ham}, we get
$$
\Phi_t^{i q_t \circ \Phi_t^{-i  v^{(\ell)}_t(x)} } = \Phi_t^{ i (\frac{|\xi|^2}2 +  v(x)  + B x\cdot \xi )  }.
$$
The exponential map begin injective near the origin, we deduce that, \eqref{New fact} is equivalent to 
\begin{equation}
\label{really to solve}
q_t \circ \Phi_t^{-i  v^{(\ell)}_t(x)} = \frac{|\xi|^2}2 +  v(x)  + B x\cdot \xi.
\end{equation}
However $\Phi_t^{-i  v^{(\ell)}_t(x)}$ is a shear transform
 $$
 \Phi_t^{-i  v^{(\ell)}_t(x)}(x,\xi) = (x, \xi + 2 t V^{(\ell)}_t).
 $$
 Consequently, decomposing naturally $q_t$, with notations similar to the ones of the iterative method, as
 $$
  q_t(x,\xi)  =  \widetilde{v}_{t}^{(m)} (x)  +  
				(\widetilde{L}_{t}x+\widetilde{U}_{t}x + t D_{t})x\cdot \xi  + \widetilde{a}_{t}(\xi)	
 $$
 we deduce that \eqref{really to solve} is equivalent to the system of equations
 \begin{align*}
 \widetilde{a}_{t}(\xi) &=  |\xi|^2 /2, \ \ \ \widetilde{L}_{t}  = L, \ \ \ \widetilde{U}_{t}  = U, \ \ V^{(\ell)}_t =  -D_{t} /2, \\ \widetilde{v}_{t}^{(m)}(x) &=  v(x) -  (\widetilde{L}_{t}x+\widetilde{U}_{t}x + t D_{t})x \cdot 2t V^{(\ell)}_t x- 2 t^2 |V^{(\ell)}_t x|^2.
 \end{align*}
Noting that by substitution, the last equation writes 
  $$
 \widetilde{v}_{t}^{(m)}(x) =  v(x) + \frac{ t}2  [B,D_t]x \cdot x + \frac{ t^2}2 D_{t}x\cdot D_{t} x,
 $$
 we conclude that to get the factorization \eqref{New fact}, it is enough to choose $V^{(\ell)}_t =  -D_{t} /2$ and to solve the nonlinear equation
 \begin{align*}
 F(t,a_t,L_t,U_t,v_t^{(m)}):&= (\widetilde{a}_{t}(\xi),\widetilde{L}_{t},\widetilde{U}_{t},\widetilde{v}_{t}^{(m)}(x)-\frac{ t}2  [B,D_t]x \cdot x - \frac{ t^2}2 D_{t}x\cdot D_{t}) \\ &=  (|\xi|^2 /2,L,U,v(x)).
 \end{align*}
 
 As in the proof of Theorem \ref{thm_split_Lie}, this nonlinear equation can be solved by the Implicit Function Theorem. Indeed, $F(0,|\xi|^2 /2,L,U,v(x)) = (|\xi|^2 /2,L,U,v(x))$, $F$ is an analytic function of $t,a_t,L_t,U_t,v_t^{(m)}$, and the partial differential of $F$ in $(0,|\xi|^2 /2,L,U,v(x))$ with respect to $(a_t,L_t,U_t,v_t^{(m)})$ is the identity. Note that the natural iterative method associated with the resolution of this nonlinear equation with the implicit function theorem is exactly the one introduced above.
 %substep
 %$$
 %\widetilde{v}_{t}^{(m)}(x) =  v(x) -  2 t Bx \cdot V^{(\ell)}_t x - 2 t^2 D_{t}x\cdot V^{(\ell)} x - 2 t^2 |V^{(\ell)}_t x|^2.
 %$$

\end{proof}

\subsection{Fokker--Planck equations}
\label{sub_FP}
We apply our exact splitting methods to two Fokker--Planck equations. These equations can be used to describe particles system with collisions (in plasma physics or astrophysics). We refer to \cite{HSH, DHL, AB} for more details about these equations.

\subsubsection{Exact splitting for the Fokker--Planck equation} First, we focus on the classical inhomogeneous Fokker--Planck equation
\begin{equation}
\label{FP}
\tag{FP}
\partial_t u(t,x,v) + v\partial_x u(t,x,v) = \partial_v (v+ \partial_v) u(t,x,v).
\end{equation}

Note that, since this equation comes from kinetic models, the variable are not denoted $x_1,x_2$ but $x,v$. Implicitly, in this paper, the Fourier variable canonically associated with $v$ is denoted by $\eta$. 

\begin{proposition}
\label{prop_NFP}
The following factorization provides an exact splitting for \ref{FP}
\begin{equation*}
\forall t\geq 0, \ e^{- t ( v\partial_x - \partial_v^2 - \partial_v v  )} = e^{t/2}  e^{- (e^t-1) v\partial_x } e^{ \nabla \cdot A_t \nabla} e^{ i\alpha_t  \partial_v^2}  e^{-i \beta_t  v^2} e^{-i \beta_t \partial_v^2} e^{i \alpha_t  v^2}.
\end{equation*}
where   $\alpha_{t} = \frac12 \sqrt{(1 - e^{-t})e^{-t}}$, $\beta_t = \frac12 \sqrt{e^t-1}$, $\nabla = \transp{ (\partial_x,\partial_v )}$ and $A_t$ is the positive matrix defined by
$$
A_t =\frac12 \begin{pmatrix} e^{2t} + 2t + 3 -4 e^t  & -4 \sinh^2(t/2) \\
 -4 \sinh^2(t/2) & 1-e^{-2t}
\end{pmatrix}.
$$
\end{proposition}
A priori, it could seem strange to have to solve some Schr\"odinger equations in order to compute the solution of \ref{FP}. However a part of the Fokker--Planck operator is associated with a dilatation, i.e. $v\partial_v$, and as explained in the previous subsection (see Proposition \ref{prop_dilat}), the semigroup it generates cannot be computed only by compositions of shear transforms.

\begin{proof}
First, let us check that $A_t$ is nonnegative. Since the second diagonal coefficient of $A_t$ is positive, to check that $A_t$ is nonnegative, it is enough to prove that $\det A_t\geq0$ for $t>0$. By a some elementary formal computations, we get
$$
\det A_t =   \frac{(3e^{4t} - 12 e^{3t} + (8t+2)e^{2t} + 20 e^t - 8t - 13)e^{-2t}}{16}.
$$
Consequently, we just have to prove that $(16 e^{2t} \det A_t) \geq 0$. However, realizing the Taylor expansion of $16 e^{2t} \det A_t$ in $0$ (with a formal computation software), we get 
$$16 e^{2t} \det A_t \mathop{=}_{t\to 0} \mathcal{O}(t^4).$$ 
Thus, it is enough to prove that
$$
\partial_t^2 (e^{-2t}\partial_t^2 (16 e^{2t} \det A_t)) \geq 0,
$$
which is obvious since
$$
\partial_t^2 (e^{-2t}\partial_t^2 (16 e^{2t} \det A_t)) = 192e^{2t} - 108 e^{t} + 20 e^{-t} \geq 0.
$$
Actually, it would be possible to prove a priori that $A_t$ is nonnegative. Indeed, in the Fourier variables, the semigroup of \ref{FP} is associated with a transport equation for which many elementary explicit computations can be done. In particular, it follows of a formula of Kolmogorov (see \cite{Ko34,ABfou}) that $A_t$ admits an integral representation on which it is obvious to see that $A_t$ is nonnegative.

Then, we consider the following factorization (whose form could be guessed using Theorem \ref{thm_split_Lie}) that can be checked easily using a formal computation software
$$
\forall t \geq 0, \ \Phi_t^{\eta^2  + i(v\xi-v \eta)} = \Phi_1^{ i(e^t-1)v\xi}  \Phi_1^{ (\xi,\eta) A_t \transp{(\xi,\eta)} } \Phi_t^{- i v\eta}.
$$

Since $A_t$ is nonnegative for $t\geq 0$, applying Proposition \ref{prop_key}, we deduce that
$$
\forall t\geq 0, \ e^{- t ( v\partial_x - \partial_v^2 - \partial_v v  )} =  e^t e^{- (e^t-1) v\partial_x } e^{ \nabla \cdot A_t \nabla} e^{ t v\partial_v }.
$$
The last factor being a dilatation, we conclude this proof applying Proposition \ref{prop_dilat}.
\end{proof}

\subsubsection{Exact splitting for the Kramer--Fokker--Planck equation} $\empty$ \\
Now we focus on the Kramer--Fokker--Planck equation
\begin{equation}
\label{KFP}
\tag{KFP}
\partial_t u(t,x,v) + v^2 u(t,x,v) -\partial_v^2 u(t,x,v)+v\partial_x u(t,x,v) =0.
\end{equation}
For some discussions about this equation, we refer to \cite{HSH, AB}.

\begin{proposition}
The following factorization provides an exact splitting for \ref{KFP}
$$
\forall t\geq 0, \ e^{-t(v^2-\partial_v^2+v\partial_x)} = e^{-\frac12 \tanh t \ v^2}  e^{ \nabla\cdot A_t \nabla } e^{- \tanh t \ v\partial_x} e^{-\frac12 \tanh t \ v^2}
$$
where $A_t$ is the nonnegative matrix defined by
$$
A_t = \frac12 \begin{pmatrix} \alpha_t & \sinh^2 t\\
\sinh^2 t & \sinh 2t
\end{pmatrix} , \ \mathrm{with} \ \alpha_t = \frac12 (t-(\tanh t) (1- \sinh^2 t) ).
$$
\end{proposition}
\begin{proof} Applying Proposition \ref{prop_key}, we just have to check that $A_t$ is non negative for $t\geq 0$ and that
$$
\forall t \geq 0, \ \Phi_t^{v^2+\eta^2 + i v\xi} = \Phi_1^{\frac12 \tanh t \ v^2}  \Phi_1^{ (\xi,\eta) A_t \transp{(\xi,\eta)} } \Phi_1^{ i \tanh t \ v\xi} \Phi_1^{\frac12 \tanh t \ v^2}.
$$
This factorization, whose form can be deduced a priori using Theorem \ref{thm_split_Lie}, can be verified by a formal computation software. Since the second diagonal coefficient of $A_t$ is positive, to check that $A_t$ is non negative, it is enough to prove that $\det A_t\geq0$ for $t>0$. Thus, we conclude this proof observing that by Jensen's inequality and some elementary trigonometric computations
$$
\det A_t =  \frac14 \tanh(t) (t - \tanh(t)) \geq 0.
$$

\end{proof}

\section{Appendix}

\subsection{Proof of Lemma \ref{lem_in_fact_its_triangular}}
\label{pippo}

Here, the proof relies essentially on computations by block requiring to introduce a more convenient basis than
 the canonical basis of $\mathbb{C}^{2n+2}$ denoted by $e_1,\dots,e_{2n+2}$. This basis, denoted $\mathscr{B}$, is just a permutation of the canonical basis and is defined by 
  \begin{equation}
  \label{good_basis}
  \mathscr{B} = (e_1,\dots,e_n,e_{n+2},\dots,e_{2n},e_{2n+1},e_{n+1},e_{2n+2}).
  \end{equation}
In this basis the matrix of $J_{2n+2}$ is
$$
\mathrm{mat}_{\mathscr{B}} \ J_{2n+2} = \begin{pmatrix} J_{2n} & \\ & J_2
\end{pmatrix}
$$
and the matrix of $\mathbb{P}p_j$ is
$$
\mathrm{mat}_{\mathscr{B}}\ \mathbb{P}p_j = \begin{pmatrix} Q^{(j)} &  \frac12 \transp L^{(j)} & \\
											      \frac12 L^{(j)} & c_j &\\
											      & & 0
\end{pmatrix},
$$
Consequently, in this basis the matrix of the Hamiltonian map is
$$
\mathrm{mat}_{\mathscr{B}}\ J_{2n+2} \mathbb{P}p_j = \begin{pmatrix} J_{2n}Q^{(j)} & \frac12 \Upsilon_{q_j} J_{2n} \transp L^{(j)} & \\
												 & 0 &  \\
											      -\frac12  L^{(j)} & -c_j & 0
											      
\end{pmatrix}.
$$
Here, it is really relevant to observe the double triangular structure of this matrix. The four block on the top left corner defines an upper triangular matrix by blocks, whereas, considering these four blocks as a single block, the global matrix is lower triangular by blocks.

Now observe, through the power expansion series, that if $\Psi$ is an entire function and 
$$M = \begin{pmatrix} A & B \\ & 0 \end{pmatrix}$$
 is an upper triangular matrix by blocks then
$$
\Psi(M) = \begin{pmatrix} \Psi(A) & \left( \frac{\Psi(z)-\Psi(0)}z \right) (A) B \\ & \Psi(0) \end{pmatrix}.
$$

Consequently, we have
$$
\mathrm{mat}_{\mathscr{B}} \ \Phi_1^{\mathbb{P}p_{j}} = \begin{pmatrix} \Phi_1^{q_{j}} &-i\Upsilon_{q_j} J_{2n} \transp L^{(j)} & \\
												 & 1 &  \\
											      iL^{(j)}\Upsilon_{q_j}  & 2i\widetilde{\kappa_j}+2i c_j & 1
											      
\end{pmatrix}$$
where
$$
\widetilde{\kappa_j} = -\frac{i}2 L^{(j)} \Theta_{q_j}  J_{2n} \transp{L^{(j)}} \ \ \mathrm{with} \ \ \Theta _{q_j} = \left( \frac{e^z-1-z}{z^2} \right)(-2iJ_{2n} Q^{(j)}).
$$
At the end of the proof, we will check that $\widetilde{\kappa_j} = \kappa_j$, so for the moment assume that this relation holds.

Thus, realizing a product by block we get by a straightforward induction
$$
\mathrm{mat}_{\mathscr{B}} \ \prod_{j=1}^m \Phi_1^{\mathbb{P}p_{j}} = \begin{pmatrix} \displaystyle \prod_{j=1}^m \Phi_1^{q_{j}} &\displaystyle -i \sum_{j=1}^m \left(\prod_{k<j} \Phi_1^{q_{k}}\right)  \Upsilon_{q_j} J_{2n} \transp L^{(j)} & \\
												 & 1 &  \\
											       \displaystyle i \sum_{j=1}^m L^{(j)}\Upsilon_{q_j} \prod_{k>j} \Phi_1^{q_{k}}  & \displaystyle 2i \sum_{j=1}^m \kappa_j + c_j + \sigma_j & 1
											      
\end{pmatrix}.
$$
Identifying the blocks $(1,1),(3,1),(3,2)$ with those of $\mathrm{mat}_{\mathscr{B}} \  \Phi_1^{\mathbb{P}p_{m+1}}$ we get the system \eqref{terrible_system}. Conversely, we have to check that if \eqref{terrible_system} is satisfied then the blocks $(1,2)$ are the same. Indeed, consider a complex symplectic matrix $M\in \mathrm{Sp}_{2n}(\mathbb{C})$ with a block structure of the form
$$
\mathrm{mat}_{\mathscr{B}} \ M = \begin{pmatrix} A & B \\ & 1 \\
									    C & d & 1 \end{pmatrix}.$$
Note that since $M$ is symplectic, $M$ is invertible and consequently $A$ is also invertible. Since $M$ is symplectic, then is satisfies
\begin{equation}
\label{hoho}
M = -J_{2n} \transp M^{-1} J_{2n}.
\end{equation}
However, due to the double triangular nature of $M$, its invert can be computed easily and we have
$$
\mathrm{mat}_{\mathscr{B}} \ M^{-1} = \begin{pmatrix} A^{-1} & -A^{-1}B \\ & 1 \\
									    -CA^{-1} & CA^{-1}B-d & 1 \end{pmatrix}.
$$ 
Consequently, a straightforward block product leads to
$$
-  \mathrm{mat}_{\mathscr{B}} J_{2n} \transp M^{-1} J_{2n} = \begin{pmatrix} -J_{2n}\transp{A}^{-1}J_{2n} & -J_{2n}\transp{A}^{-1}\transp{C} \\ & 1 \\
									   - \transp{B} \transp{A}^{-1}J_{2n} & d-CA^{-1}B & 1 \end{pmatrix}.
$$
Thus, since $M$ is symplectic, we have
$$
B = -J_{2n}\transp{A}^{-1}\transp{C}.
$$
A fortiori, if two symplectic matrices have this block structure and the same top left corner blocks, if their blocks $(3,1)$ are equal then their blocks $(1,2)$ are equal. Consequently, applying this to the symplectic matrices $\Phi_1^{\mathbb{P}p_{1}} \dots \Phi_1^{\mathbb{P}p_{m}}$ and $\Phi_1^{\mathbb{P}p_{m+1}}$, we deduce that if the system \eqref{terrible_system} is satisfied then we have the factorization \eqref{the_aff_fact}.

Finally, we just have to check that $\widetilde{\kappa_j} = \kappa_j$. For this computation, we omit the indices $j$ since they are clearly irrelevant. First, we split the even indices from the odd indices in the power expansion of $\widetilde{\kappa}$ :
\begin{multline*}
2i\widetilde{\kappa} = L \Theta_{q}  J_{2n} \transp{L} = \sum_{k\in \mathbb{N}}\frac1{(k+2) !} L (-2iJ_{2n} Q)^k J_{2n} \transp{L} \\
= \sum_{k \in \mathbb{N}}\frac1{(2k+2) !} L (-2iJ_{2n} Q)^{2k} J_{2n} \transp{L}  + \sum_{k\in \mathbb{N}}\frac1{(2k+3) !} L (-2iJ_{2n} Q)^{2k+1} J_{2n} \transp{L} \\
:= \Sigma_{even} + \Sigma_{odd}.
\end{multline*}
Observing that 
$$
(J_{2n}Q )^{2k}J_{2n} = (J_{2n} Q)^{k} J_{2n} (Q J_{2n} )^{k},
$$
since $J_{2n}$ is skew-symmetric we have
$$
L (J_{2n}Q )^{2k}J_{2n}  \transp{L} = (-1)^k L  (J_{2n} Q)^{k}   J_{2n}  \transp{L  (J_{2n} Q)^{k}} =0.
$$
Consequently, $\Sigma_{even}$ vanishes. Similarly, since we have
$$
(J_{2n}Q )^{2k+1}J_{2n} = (J_{2n}Q)^k J_{2n} Q (J_{2n}Q)^k J_{2n} = (-1)^{k+1} (J_{2n}Q)^k J_{2n} Q \transp{((J_{2n}Q)^k J_{2n} )},
$$
we get
$$
\Sigma_{odd} = 2i \sum_{k \in \mathbb{N}} \frac{4^k}{(2k+1)!} L(J_{2n}Q)^k J_{2n} Q \transp{((J_{2n}Q)^k J_{2n} )} \transp L = 2i \kappa.
$$

\subsection{Proof of Theorem \ref{thm_universal}}
\label{sub_proof_thm_universal}
Before proving Theorem \ref{thm_universal}, let us to prove some preparatory lemmas.
\begin{lemma}
\label{calc_meta}
If $L$ is a bounded operator on $L^2(\mathbb{R}^n)$ such that there exists $T\in \mathrm{Sp}_{2n}(\mathbb{R})$ satisfying $L=\pm \mathscr{K}(T)$, then $L$ can be computed by an exact splitting.
\end{lemma}
\begin{proof} Let $G$ be the group generated by $\Phi_t^{i x_j^2},\Phi_t^{i \xi_j^2},\Phi_t^{i x_j \xi_k}$ for $t\in \mathbb{R}, j,k\in \llbracket1,n \rrbracket$. Applying Theorem \ref{merci_lars}, if we prove that $G= \mathrm{Sp}_{2n}(\mathbb{R})$,
we deduce that $L$ is a product of operators of the form $e^{it x_j^2 },e^{it \partial_{x_j}^2},e^{t x_j \partial_{x_k}}$ (up to the sign). Thus, since Proposition \ref{prop_dilat} states that dilatations (i.e. operators of the form $e^{t x_k \partial_k}$) can be factorized similarly, we would deduce that $L$  can be computed by an exact splitting.

Consequently, we aim at proving that $G= \mathrm{Sp}_{2n}(\mathbb{R})$. First, let us prove that $G$ contains a neighborhood of the identity in $\mathrm{Sp}_{2n}(\mathbb{R})$.

Indeed, consider the map $\Psi : \mathcal{N} \to \mathfrak{sp}_{2n}(\mathbb{R})$, where $\mathcal{N}$ is  a neighborhood of the origin in $\mathfrak{sp}_{2n}(\mathbb{R})$, defined for $Q\in S_{2n}(\mathbb{R})$ such that $J_{2n}Q \in \mathcal{N}$ by 
\begin{multline}
\label{terrible_log}
\Psi(J_{2n}Q) = \\
\log\left( \prod_{j=1}^{n} \Phi^{ i A_{j,j} x_j^2}_{1/2} \prod_{1\leq j < k \leq n} \Phi^{ i A_{j,k} x_j x_k}_{1}  \prod_{j=1}^n \prod_{k=1}^{n} \Phi^{ i C_{j,k} x_j \xi_k}_{1}  \prod_{1 \leq j < k \leq n} \Phi^{ i B_{j,k} \xi_j \xi_k}_{1} \prod_{j=1}^{n} \Phi^{ i B_{j,j} \xi_j^2}_{1/2} \right),
\end{multline}
where the natural block decomposition of $Q$ is
$$
Q=\begin{pmatrix} A & C \\ \transp{C} & B
\end{pmatrix}
$$
with $A,B\in S_{n}(\mathbb{R})$ and $C\in M_n(\mathbb{R})$. Note that to prove that $\Psi$ takes its values in $\mathfrak{sp}_{2n}(\mathbb{R})$, it is enough to apply the Baker--Campbell--Hausdorff formula. 

Since the differential of the exponential in the origin and the differential of the logarithm in the identity are equal to the identity, we deduce by composition that the differential of $\Psi$ in the origin is also the identity. Thus, since $\Psi$ vanishes in the origin, we deduce of the Inverse Function Theorem that $\Psi$ defines a local homeomorphism around the origin. Furthermore, we recall\footnote{the reader could refer, for example, to Lemma 3.10 in \cite{AB} for a detailed proof.} that the exponential is an homeomorphism between a neighborhood of the origin in $\mathfrak{sp}_{2n}(\mathbb{R})$ and a neighborhood of the identity in $\mathrm{Sp}_{2n}(\mathbb{R})$.  Consequently, we deduce that each matrix in $\mathrm{Sp}_{2n}(\mathbb{R})$ close enough to the identity can be written as a product of the form of the product in the logarithm in \eqref{terrible_log}. A fortiori, we have proven that $G$ contains a neighborhood of the identity in $\mathrm{Sp}_{2n}(\mathbb{R})$. Let $V$ denotes this neighborhood.

Since $\mathrm{Sp}_{2n}(\mathbb{R})$ is connected (see e.g. the subsection 4.4 of \cite{Arnold_book}), to prove that $G= \mathrm{Sp}_{2n}(\mathbb{R})$, we just have to prove that $G$ is closed and open in $\mathrm{Sp}_{2n}(\mathbb{R})$. Indeed, if $g\in G$ then $gV$ is a neighborhood of $g$ in $\mathrm{Sp}_{2n}(\mathbb{R})$ and since $V$ is included in $G$ then $gV$ is also included in $G$. Thus $G$ is open in $\mathrm{Sp}_{2n}(\mathbb{R})$. Conversely, if $g\notin G$ then $gV$ is also a neighborhood of $g$ in $\mathrm{Sp}_{2n}(\mathbb{R})$ but since $G$ is a group we have $gV \cap G = \emptyset$. Consequently, the complementary of $G$ in $\mathrm{Sp}_{2n}(\mathbb{R})$ is open, i.e. $G$ is closed in $\mathrm{Sp}_{2n}(\mathbb{R})$, which conclude the proof.
\end{proof}

\begin{lemma}
\label{split_aff_im} If $\ell:\mathbb{R}^{2n}\to \mathbb{R}$ is a real linear form then $\exp(i \ell^w)$ can be computed by an exact splitting.
\end{lemma}
\begin{proof}
Let $L\in \mathbb{R}^{2n}$ be the matrix of $\ell$ and let $c=c_1$ where
$$
c_t =  \frac{t}2  \sum_{j=1}^{n} L_j L_{j+n}.  
$$ 
Applying Lemma \ref{lem_in_fact_its_triangular}, we have
$$
\forall t\in \mathbb{R}, \ \ \Phi^{-i \mathbb{P} c_t }_t \prod_{j=1}^n \Phi^{-i  L_j \mathbb{P} x_j }_t  \prod_{j=1}^n \Phi^{-i  L_{j+n} \mathbb{P} \xi_j }_t = \Phi^{-i  \mathbb{P} \ell}_t.
$$
Consequently, applying Proposition \ref{prop_key} at $t=1$, we get
$$
e^{i \ell^w} = e^{i c_t} \prod_{j=1}^n e^{i L_j x_j} \prod_{j=1}^n e^{ L_j \partial_{x_j} }.
$$
\end{proof}

\begin{lemma} 
\label{lem_split_hom}If $p:\mathbb{R}^{2n}\to \mathbb{R}$ is a real valued polynomial of degree $2$ or less then there exists a real linear form $\ell:\mathbb{R}^{2n}\to \mathbb{R}$ and $c\in \mathbb{R}$,  such that
\begin{equation}
\label{res_lem_split_hom}
e^{ip^w} =e^{ic} e^{iq^w} e^{i\ell^w},
\end{equation}
where $q$ is the quadratic part of $p$.
\end{lemma}
\begin{proof} Let $l$ be the linear part of $p$. Considering the natural action of the entire functions on $M_{2n}(\mathbb{C})$, let $\ell=\ell_1$ and $c=c_1$ where
$$
\ell_t = l \circ \left(\frac{e^z-1}z \right) (- 2 t J_{2n} Q) \ \mathrm{and} \ c_t= p(0) - t^2 \sum_{k\in \mathbb{N}} \frac{(-4)^k t^{2k} }{(2k+3)!} \ q( (J_{2n} Q)^k J_{2n} \transp L)
$$
with $Q\in S_{2n}(\mathbb{R})$, the matrix of $q$ and $L$ the matrix of $l$. Applying Lemma \ref{lem_in_fact_its_triangular}, we have
$$
\forall t\in \mathbb{R}, \ \ \Phi^{-i \mathbb{P} c_t }_t \Phi^{-i  \mathbb{P} q}_t  \Phi^{-i \mathbb{P} \ell_t }_t = \Phi^{-i  \mathbb{P} p}_t.
$$
Consequently, applying Proposition \ref{prop_key} at $t=1$, we get \eqref{res_lem_split_hom}.
\end{proof}

\begin{lemma} \label{lem_translat}If $p:\mathbb{R}^{2n}\to \mathbb{R}$ is a real valued polynomial of degree $2$ or less bounded below, then there exists a real valued linear form $\ell :\mathbb{R}^{2n}\to \mathbb{R}$ and $c\in \mathbb{R}$ such that
\begin{equation}
\label{res_lem_translat}
e^{-p^w} = e^{-c}  e^{-i\ell^w} e^{-q^w} e^{i\ell^w},
\end{equation}
where $q$ is the quadratic part of $p$.
\end{lemma}
\begin{proof} 
Applying Lemma \ref{lem_struc_pos}, we get $Y\in \mathbb{R}^{2n}$ such that $p=p_1$ where
$$
p_t = q(\cdot \ - \ tY)+p(Y).
$$
and $X=(x_1,\dots,x_n,\xi_1,\dots,\xi_n)$.

Let $\ell = -\transp{(J_{2n}Y)}X$ and $c=p(Y)$. Let $\mathscr{B}$ be the basis introduce in the proof of Lemma \ref{lem_in_fact_its_triangular} and defined by \eqref{good_basis}.
Observing that 
$$
\mathrm{mat}_{\mathscr{B}} J_{2n+2}\nabla \mathbb{P}\ell = \begin{pmatrix} x_{n+1} Y \\ 0 \\ \ell \end{pmatrix},
$$
we deduce that
$$
(\mathrm{mat}_{\mathscr{B}} \Phi^{-i \mathbb{P}\ell}_t) \begin{pmatrix} X \\ x_{n+1} \\ \xi_{n+1}\end{pmatrix} =  \begin{pmatrix} X - t x_{n+1}  Y \\ x_{n+1} \\ \xi_{n+1} - t \ell \end{pmatrix}.
$$
Consequently, we have
$$
\mathbb{P}(q + c ) \circ \Phi^{-i \mathbb{P}\ell}_t = \mathbb{P}p_t. 
$$
However, $\Phi^{-i \mathbb{P}\ell}_t$ is a symplectic map, so we have
$$
\Phi^{i \mathbb{P}\ell}_t \Phi_t^{ \mathbb{P}(q + c)}\Phi^{-i \mathbb{P}\ell}_t = \Phi_t^{ \mathbb{P}(q + c)\circ \Phi^{-i \mathbb{P}\ell}_t} = \Phi_t^{ \mathbb{P}p_t}.
$$
Now observing that the Hamiltonian $\mathbb{P} c$ commutes (i.e. for the canonical Poisson bracket) with all the other Hamiltonians, we deduce of the Noether theorem (or of Lemma \ref{lem_in_fact_its_triangular}) that
$$
\forall t\in \mathbb{R}, \ \Phi_t^{ \mathbb{P}c}  \Phi^{i \mathbb{P}\ell}_t \Phi_t^{ \mathbb{P}q }\Phi^{-i \mathbb{P}\ell}_t = \Phi_t^{ \mathbb{P}p_t}.
$$
Consequently, applying Proposition \ref{prop_key} at $t=1$, we get \eqref{res_lem_translat}.
\end{proof}

\begin{lemma}
\label{lem_pos_quad_comp} If $q:\mathbb{R}^{2n}\to \mathbb{R}$ is a non-negative real quadratic form then $\exp(-q^w)$ can be computed by an exact splitting.
\end{lemma}
\begin{proof}
Applying Theorem 21.5.3 of \cite{Lars_book}, we get a symplectic change of coordinates $T\in \mathrm{Sp}_{2n}(\mathbb{R})$ such that there exists $m\in \llbracket 1,n \rrbracket$ and $0<\lambda_1\leq \dots \leq \lambda_m$ some real numbers such that
$$
q\circ T = \sum_{j=1}^m \lambda_j(x_j^2 + \xi_j^2) + \sum_{j=m+1}^n x_j^2.
$$
Consequently, since $T$ is symplectic, applying Noether theorem, we have
$$
 \Phi^q_1 = T \Phi^{q\circ T}_1 T^{-1} = T \left(\prod_{j=1}^m \Phi^{x_j^2 + \xi_j^2}_{\lambda_j} \prod_{j=m+1}^n \Phi^{x_j^2}_1 \right) T^{-1}.
$$
Applying Theorem \ref{merci_lars}, we get, at the level of the Fourier Integral Operators,
$$
\pm e^{-tq^w}= \mathscr{K}(T) \left( \prod_{j=1}^m  e^{-\lambda_j (x_j^2-\partial_{x_j}^2) }   \prod_{j=m+1}^n e^{-x_j^2} \right)   \mathscr{K}(T^{-1}).
$$
Recalling that, as a consequence of the formula \eqref{split_exact_harm_osci}, the semigroups $e^{\lambda_j (x_j^2-\partial_{x_j}^2) }$ can be computed by exact splittings, we deduce of Lemma \ref{calc_meta} that the same applies for $e^{-tq^w}$. 
\end{proof}

Now, we can prove Theorem  \ref{thm_universal}.
\begin{proof}[Proof of Theorem  \ref{thm_universal}] Let $p$ be a polynomial of degree $2$ or less on $\mathbb{C}^{2n}$ whose real part is bounded below on $\mathbb{R}^{2n}$. We aim at proving that $e^{-p^w}$ can be computed by an exact splitting in the sense of Definition \ref{def_comput}. 

Applying Corollary \ref{cor_proof_proj_pos}, we get a constant $c\in \mathbb{R}$ such that $\mathbb{P}(p-c)\geq 0$. Then applying Theorem 2.1 and Lemma 3.10 of \cite{AB}, we get $t_0>0$ and two real quadratic forms $a_t,b_t:\mathbb{R}^{2n+2} \to \mathbb{R}$ depending analytically on $t\in (-t_0,t_0)$, $a_t$ being non-negative, and such that if $|t| < t_0$ then
$$
\Phi_t^{a_t} \Phi_t^{i b_t} = \Phi_t^{\mathbb{P}(p-c)}. 
$$
Furthermore, it follows of the Baker-Campbell-Hausdorff formulas and formulas of (3.22) and (3.41) of \cite{AB} defining $a_t,b_t$ that these quadratic forms belong to the complex Lie algebra generated by $\mathbb{P}(p-c)$ and $\mathbb{P}(\overline{p}-c)$. Observing that for all polynomials $p_1,p_2$ of degree $2$ or less on $\mathbb{C}^{2n}$ we have
$$
\{ \mathbb{P}p_1,\mathbb{P}p_2 \} = \mathbb{P}\{p_1,p_2\},
$$
the image of $\mathbb{P}$ is a Lie algebra. Consequently, $a_t,b_t$ belong to the image of $\mathbb{P}$, i.e. there exist two real polynomials of degree $2$ or less on $\mathbb{R}^{2n}$ depending analytically on $t$ and denoted $p_{t}^{(r)}$ and $p_{t}^{(i)}$ such that
$$
a_t = \mathbb{P}p_{t}^{(r)} \ \mathrm{and} \ b_t = \mathbb{P}p_{t}^{(i)}.
$$
Note that since $p_{t}^{(r)}$ is the restriction of $a_t $ on the affine subspace $\{(x_{n+1},\xi_{n+1})=(1,0)\}$, it is also non-negative.

Now applying Proposition \ref{prop_key}, we deduce that if $0\leq t<t_0$ we have
\begin{equation}
\label{affi_polar}
e^{-t p^w} = e^{-tc} e^{-t (p_{t}^{(r)})^w} e^{-it (p_{t}^{(i)})^w}.
\end{equation}
Let $t_\star \in (0,t_0)$ be such that there exists $n\in \mathbb{N}$ satisfying $t_\star=n^{-1}$. Since $(e^{-t p^w})_{t\geq 0}$ is a semigroup, we have
$$
(e^{-t_{\star} p^w})^n =  e^{-nt_{\star} p^w} = e^{-p^w}.
$$
Consequently, if $e^{-t_{\star} p^w}$ can be computed by an exact splitting then the same applies for $e^{-p^w}$. Furthermore, from the factorization \eqref{affi_polar}, we deduce that if $e^{-t_{\star} (p_{t_{\star}}^{(r)})^w}$ and $e^{-it_{\star} (p_{t_{\star}}^{(i)})^w}$ can be computed by an exact splitting then the same applies for $e^{-t_{\star} p^w}$. Consequently, we just have to focus on these two semigroups.

On the one hand, applying Lemma \ref{lem_split_hom}, Lemma \ref{split_aff_im}, Lemma \ref{calc_meta} and Theorem \ref{merci_lars} (to justify that the semigroup generated by a quadratic differential operator is a Fourier Integral Operator), we deduce that $\exp(-it_{\star} (p_{t_\star}^{(i)})^w)$ can be computed by an exact splitting.

On the other hand, applying Lemma \ref{lem_translat}, Lemma \ref{split_aff_im} and Lemma \ref{lem_pos_quad_comp}, we deduce  that $\exp(-t_{\star} (p_{t_{\star}}^{(r)})^w)$ can be computed by an exact splitting.
\end{proof}

\subsection{Proof of Lemma \ref{lem_pos}}
\label{sub_proof_lem_pos} We aim at proving that if $\Re q\geq 0$ on $\mathbb{R}^{2n}$ then $\Phi_t^q\equiv e^{-2itJ_{2n}Q} \in \mathrm{Sp}_{2n}^+(\mathbb{C})$, where $Q$ is the matrix of $q$. First, we recall that from Proposition \ref{prop_elementary_ham} we know that $\Phi_t^q$ is a symplectic transformation (i.e. $\Phi_t^q \in \mathrm{Sp}_{2n}(\mathbb{C})$).

So we aim at proving that $\Phi_t^q$ is nonnegative, i.e.
\begin{equation}
\label{pos_app}
\forall X \in \mathbb{C}^{2n}, \ \transp{\overline{X}}\transp{\overline{\Phi_t^q}}(-iJ_{2n}) \Phi_t^q X - \transp{\overline{X}}(-iJ_{2n})  X \in \mathbb{R}_+.
\end{equation}
Since $\Phi_0^q=I_{2n}$, we have
\begin{multline*}
\transp{\overline{\Phi_t^q}}(-iJ_{2n}) \Phi_t^q - (iJ_{2n}) = \int_{0}^t \partial_s \left( \transp{\overline{\Phi_s^q}}(-iJ_{2n}) \Phi_s^q \right) \mathrm{d}s =  \int_{0}^t \left(e^{-2is\bar{Q}J_{2n}} (-iJ_{2n})  e^{-2isJ_{2n}Q} \right) \mathrm{d}s \\
=   \int_{0}^t \left(e^{-2is\bar{Q}J_{2n}} \left[ (-2 i \bar{Q}J_{2n}) (-iJ_{2n}) + (-iJ_{2n}) (-2 i J_{2n}Q) \right]   e^{-2isJ_{2n}Q} \right) \mathrm{d}s\\ = 4  \int_{0}^t  \transp{\overline{\Phi_s^q}} (\Re Q ) \Phi_s^q \  \mathrm{d}s.
\end{multline*}
Since $\Re Q$ is a real symmetric nonnegative matrix, $\transp{\overline{\Phi_s^q}} (\Re Q ) \Phi_s^q$ is a Hermitian nonnegative matrix and thus $\transp{\overline{\Phi_t^q}}(-iJ_{2n}) \Phi_t^q - (iJ_{2n})$ is also a Hermitian nonnegative matrix which proves \eqref{pos_app}.

\end{document}